\documentclass{article}
\usepackage{graphicx}
\usepackage{hyperref}
\usepackage[english]{babel}
\usepackage{amssymb,amsmath, amsthm, amscd,ifthen}
\usepackage{bm}
\usepackage{xcolor}
\usepackage{thmtools} 
\usepackage{thm-restate} 
\usepackage{cleveref} 

\newcommand{\m}{\mathcal}

\newtheorem*{thm*}{Theorem}
\newcommand{\ff}{{\mathcal F}}
\newcommand{\aaa}{{\mathcal A}}
\newcommand{\pp}{{\mathcal P}}
\newtheorem*{cla*}{Claim}

\newtheorem{thm}{Theorem}
\newtheorem{gypo}{Conjecture}
\newtheorem{opr}{Definition}

\newtheorem{cla}[thm]{Claim}

\newtheorem{theorem}[thm]{Theorem}

\newtheorem{corollary}[thm]{Corollary}
\newtheorem{lemma}[thm]{Lemma}

\date{}

\newtheorem{defn}{Definition}

\DeclareMathOperator{\E}{\mathrm E}

\title{A complete solution of the Erd\H os--Kleitman matching problem for $n\le 3s$.}
\author{Andrey Kupavskii, Georgy Sokolov}

\begin{document}

\maketitle


\begin{abstract} Given integers $n\ge s\ge 2$, let $e(n,s)$ stand for the maximum size of a family  $\ff$ of subsets of an $n$-element set that contains no $s$ pairwise disjoint members. The study of this quantity goes back to the 1960s, when Kleitman determined 
$e(sm-1,s)$ and $e(sm,s)$ for all integer $m,s\ge 1$. The question of determining $e(n,s)$ is closely connected to its uniform counterpart, the subject of the famous Erd\H os Matching Conjecture. 

The problem of determining $e(n,s)$ has proven to be very hard and, in spite of some progress during these years, even a general conjecture concerning the value of $e(n,s)$ is missing. In this paper, we completely solve the problem for $n\le 3s$. In this regime, the average size of a set in an $s$-matching is at most $3$, and it is a delicate interplay between the `missing' $2$- and $3$-element sets that plays a key role here. {\it Four} types of extremal families appear in the characterization. Our result sheds light on how the extremal function $e(n,s)$ may behave in general. 
\end{abstract}

\section{Introduction}
Let $[n] := \{1,2,\ldots, n\}$ and, more generally, $[a,b]=\{a,a+1,\ldots, b\}$. For a set $X$ and an integer $k$, let  $2^{X}$, ${X\choose k}$ and ${X\choose \geq k}$ stand for the power set of $X$, the set of its $k$-element subsets and the set of its subsets with size at least $k$, respectively. Any collection of sets is called a {\it family.} A {\it matching} is a collection of pairwise disjoint sets. An {\it $s$-matching} is a matching of size $s$. Given a family $\ff,$ its {\it matching number}
$\nu(\mathcal F)$ is the size of the largest matching in $\ff$.

One of the classical topics in extremal set theory is the study of {\it intersecting} families, that is, families with matching number $1$. Erd\H os, Ko and Rado~\cite{EKR} showed that the largest intersecting family $\ff\subset 2^{[n]}$ has size at most $2^{n-1}$, and that for $n\ge 2k$ the largest intersecting family $\ff\subset {[n]\choose k}$ has size ${n-1\choose k-1}.$ In the several years that followed, Erd\H os asked for the size of the largest family avoiding an $s$-matching. Let us introduce the following two quantities.
\begin{align*}
    e(n,s)&=\max\big\{|\ff|: \ff\subset 2^{[n]}, \nu(\ff)<s\big\},\\
    e_k(n,s)&=\max\Big\{|\ff|: \ff\subset {[n]\choose k}, \nu(\ff)<s\Big\}.
\end{align*}
\subsection{The uniform case}
The value of $e_k(n,s)$ is the subject of the famous Erd\H os Matching Conjecture. Let us define the families $\aaa_i^{(k)}(n,s):$
\begin{equation}\label{eq0011} \aaa_i^{(k)}(n,s-1) := \Bigl\{F\in {[n]\choose k}:|F\cap [si-1]|\ge i\Bigr\}, \ \ \ 1\le i\le k.\end{equation}
It is not difficult to see that $\nu(\aaa_i^{(k)}(n,s-1)\le s-1$ for any $n,s$.
\begin{gypo}[Erd\H os Matching Conjecture \cite{E}]\label{conj2} For $n\ge sk$
\begin{equation}\label{eq008} e_k(n,s) = \max\bigl\{|\aaa_1^{(k)}(n,s-1)|,|\aaa_k^{(k)}(n,s-1)|\bigr\}.
\end{equation}
\end{gypo}
Let us comment on the statement of the conjecture. First, the condition $n\ge sk$ is required in order for ${[n]\choose k}$ to contain an $s$-matching. Second, note that the conjecture suggests that only two out of $k+1$ families $\aaa_i$ could be extremal. This is different from some other similar-looking questions, notably, the Complete $t$-Intersection Theorem for ${[n]\choose k}$ \cite{AK}. There, all families could be extremal depending on the parameters $n,k,t.$  

The conjecture (\ref{eq008}) is known to be true for $k= 2$ \cite{EG}, and $k=3$ \cite{LM,F11}. In fact, an elegant proof of Frankl of the $k=2$ case of the conjecture, initially proven by Erd\H os and Gallai, will play an important role in this paper. We will recite it later.

The conjecture was proved for $n\ge n_0(s,k)$ by Erd\H os \cite{E}. Many researchers worked on extending the range in which the conjecture is valid:  \cite{E}, \cite{BDE}, \cite{HLS}, \cite{FLM},  \cite{F4}, \cite{FK21}. Notably, in the last two papers, it is proved that
\begin{equation}\label{eq009} e_k(n,s) = |\aaa_1|={n\choose k}-{n-s+1\choose k} \end{equation}
for $n\ge (2s+1)k-s$ and for ($n\ge \frac 53sk$, $s\ge s_0$), respectively. In \cite{F7} and \cite{KoKu}, the authors studied the problem from the other end and showed that $\aaa_k$ is extremal for ($n\le s(k+1/100k)$, $s\ge ck^3$ with some absolute $c>0$). We refer the reader to \cite{aletal, FK21} for the connections of the Erd\H os Matching Conjecture and other questions, such as Dirac thresholds and small deviations in probability theory. In \cite{HLS}, \cite{K49}, the multi-family variant of the EMC was addressed.  In \cite{FK6}, a Hilton--Milner type stability result for the EMC is obtained.

\subsection{The non-uniform case}
The study of $e(n,s)$ was also initiated by Erd\H os at around the same time. The behavior of $e(n,s)$ heavily depends on $n\ ({\rm mod\ } s)$.  Answering a question of Erd\H os, Kleitman proved the following theorem. 
\begin{thm}[Kleitman \cite{Kl}]\label{thmkl}
\begin{align}\label{eq001} e(sm-1,s) &= \sum_{t=m}^{sm-1}{sm-1\choose t},\\
\label{eq002} e(sm,s) &= {sm-1\choose m}+\sum_{t=m+1}^{sm}{sm\choose t}.
\end{align}
\end{thm}
The matching example for the first case is the family ${[n]\choose \ge m}$ of all subsets of $[n]$ of size at least $m$. It is also  not difficult to see that $e(sm,s) = 2e(sm-1,s)$. In general, $e(n+1,s)\ge 2e(n,s)$ because of the {\it doubling} construction. Given a family $\ff\subset 2^{[n]}$ with $\nu(\ff)<s$, we may construct the {\it doubling} $\ff'$ of $\ff$ as follows: $\ff' = \{F\subset [n+1]: F\cap [n]\in \ff\}$. It is easy to see that $\nu(\ff')<s$ as well.

In \cite{FK9}, Frankl and the first author put forth the following conjecture.
\begin{opr}\label{def3} Let $n = sm+s-\ell$, $0<\ell\le s$. Set
\[\pp(s,m,\ell) := \bigl\{P\subset 2^{[n]}: |P|+|P\cap [\ell-1]|\ge m+1\bigr\}.\]
\end{opr}
It is not difficult to check (cf. \cite{FK9}) that $\nu( \m P(s,m,\ell))<s.$
\begin{gypo}[\cite{FK9}]\label{conj1}
Suppose that $s\ge 2, m\ge 1$, and $n = sm+s-\ell$ for some integer $0<\ell\le \lceil \frac s2\rceil$. Then
\begin{equation}\label{eq007} e(sm+s-\ell,s) = |\pp(s,m,\ell)|.
\end{equation}
\end{gypo}
They confirmed the conjecture for a variety of cases:
\begin{thm}[\cite{FK9,FK8}]\label{thmfk} \ $e(sm+s-\ell,s) = |\pp(s,m,\ell)|$ holds for
\begin{align*} &\mathrm{(i)}\ \ \ \ \ell = 2, \\
    &\mathrm{(ii)}\ \ \  m=1,\\
     &\mathrm{(iii)}  \ \  s\ge \ell m+3\ell+3.
\end{align*}
\end{thm}
The case of $\ell=2,s=3$ was solved much earlier by Quinn \cite{Q} in his PhD dissertation. We also note that a sum-type inequality for $s$ families avoiding a matching is much easier to obtain, see \cite{FK6}.

Let us return to the definition of $\m P(s,m,\ell)$. The two main uniformities to understand the construction are $m$ and $m+1$. One sees that the family contains all sets of size $m+1$, as well as all sets of size $m$ that intersect $[\ell-1]$, i.e., the family $\aaa_1^{(m)}(n,\ell-1)$. Thus, in fact, the sets of uniformity $m$ form a family with no matching of size $\ell$. We should note here that the sets in $\m P(s,m,\ell)\cap {[n]\choose \le m}$, a bulk of which is  $\aaa_1^{(m)}(n,\ell-1)$, could be replaced by any other family in ${[n]\choose m}$ with no matching of size $\ell$. As a result, we would obtain a construction of a family in $2^{[n]}$ with no $s$-matching.

Here, one sees that the problems of determining $e(n,s)$ and $e_k(n,s)$ are, in fact, closely connected. In particular, let us give  the reason for the constraint $\ell\le s/2$ in Conjecture~\ref{conj1}: in this regime, $n\ge 2\ell m,$ and the result of Frankl \cite{F4} guarantees that the largest family in ${[n]\choose m}$ with no matching of size $\ell$ is indeed $\aaa_1^{(m)}(n,\ell-1)$. 

Incidentally, one could also interpret the extremal family for $n=sm = sm+s-s$ in the same way: the extremal example consists of all sets of size at least $m+1$, as well as of the family of sets of size $m$ with no $s$-matching. The construction for the sets of uniformity $m$ is, however, different: it is $\aaa_k^{(m)}(n,s-1)$. Actually, the situation is even more complicated, as is shown in \cite{FK9}: already for $n = sm+1$ there is a construction in which some  $(m+1)$-element sets are missing and that is bigger than constructions with all $(m+1)$-element sets present. Namely, it is the doubling construction. 

Concluding the discussion, the authors of \cite{FK9} state that it is hard to formulate a general conjecture, and that they had already struggled with the case $n=2s+t$ for some small values of $t$.
\subsection{Main result}
In this paper, we find the value of $e(n,s)$ for all $n,s$, provided $n\le 3s.$ The easy case $n<2s$ is solved in Theorem~\ref{thmfk}. The cases $n=2s$ and $n=3s$ are answered in \eqref{eq002}. Thus, we are left to deal with the cases $2s<n<3s$. For the rest of the paper, unless otherwise stated, we use notation $n=2s+c=3s-\ell$, where $c, \ell\in[s-1]$, $c+\ell=s$. 

In order to state the result, we describe the prospective extremal families. We have already seen the first example. 
$$\mathcal{P}(s, \ell) := \m P(s,2,\ell)=\big\{P \in 2^{[n]} : |P| + |P \cap [\ell - 1]| \geq 3\big\}.$$
The second family also contains all sets of size at least $3$, but differs on the $2$-uniform layer. It has the other extremal example for $2$-uniform families avoiding an $\ell$-matching: a clique on $2\ell-1$ vertices.
$$\mathcal{P}'(s, \ell) := {[n] \choose \geq 3} \cup {[2\ell-1] \choose 2}.$$ 
The third family misses some $3$-element sets, and instead has a larger $2$-uniform layer. It is also a clique. $$\mathcal{Q}(s, \ell) := {[n] \choose \geq 3} \cup {[s+\ell-1] \choose 2} \setminus {[s+\ell, n] \choose 3}.$$
Finally, the fourth family is an iterated doubling of the extremal family for $n=2s-1$.
$$\mathcal{W}(s, \ell) := \big\{P \in 2^{[n]} : |P \cap [2s - 1]| \geq 2\big\}.$$ We note that none of these families contain $0$- or $1$-element sets. In the next section, we show that all these families indeed avoid $s$-matchings.

The main result of the paper is the following theorem.

\begin{theorem} \label{t.main}
	Let $n,s,\ell,c$ be positive integers such that $n=2s+c=3s-\ell$, and $c,\ell\in [s-1]$. Then $$e(n, s) = \max\big\{|\mathcal{P}(s, \ell)|, |\mathcal{P}'(\ell)|, |\mathcal{Q}(s, \ell)|, |\mathcal{W}(s, \ell)|\big\}.$$ Moreover, if $s \geq 3$, $\ff\subset 2^{[n]}$ is shifted, has no $s$-matching and $|\ff|=e(n,s)$, then $\ff$ must coincide with one of the families above.
\end{theorem}

\begin{figure}

\centering

\includegraphics[width=0.8\linewidth]{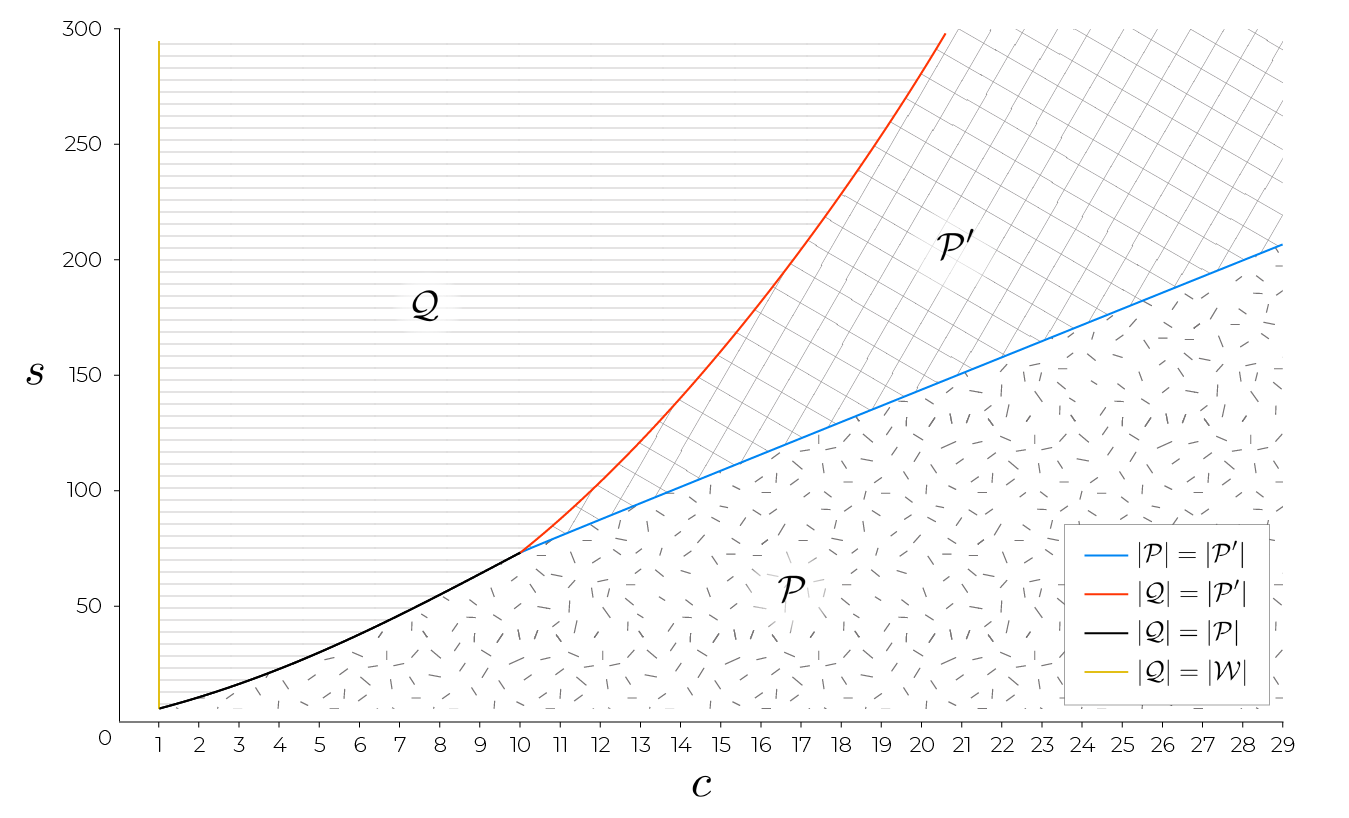}

\caption{Extremal families. $\mathcal{W}$ is one of the extremal families for $c = 1, s \geq 5$.}

\label{fig1}

\end{figure}

\begin{figure}

\centering

\includegraphics[width=0.6\linewidth]{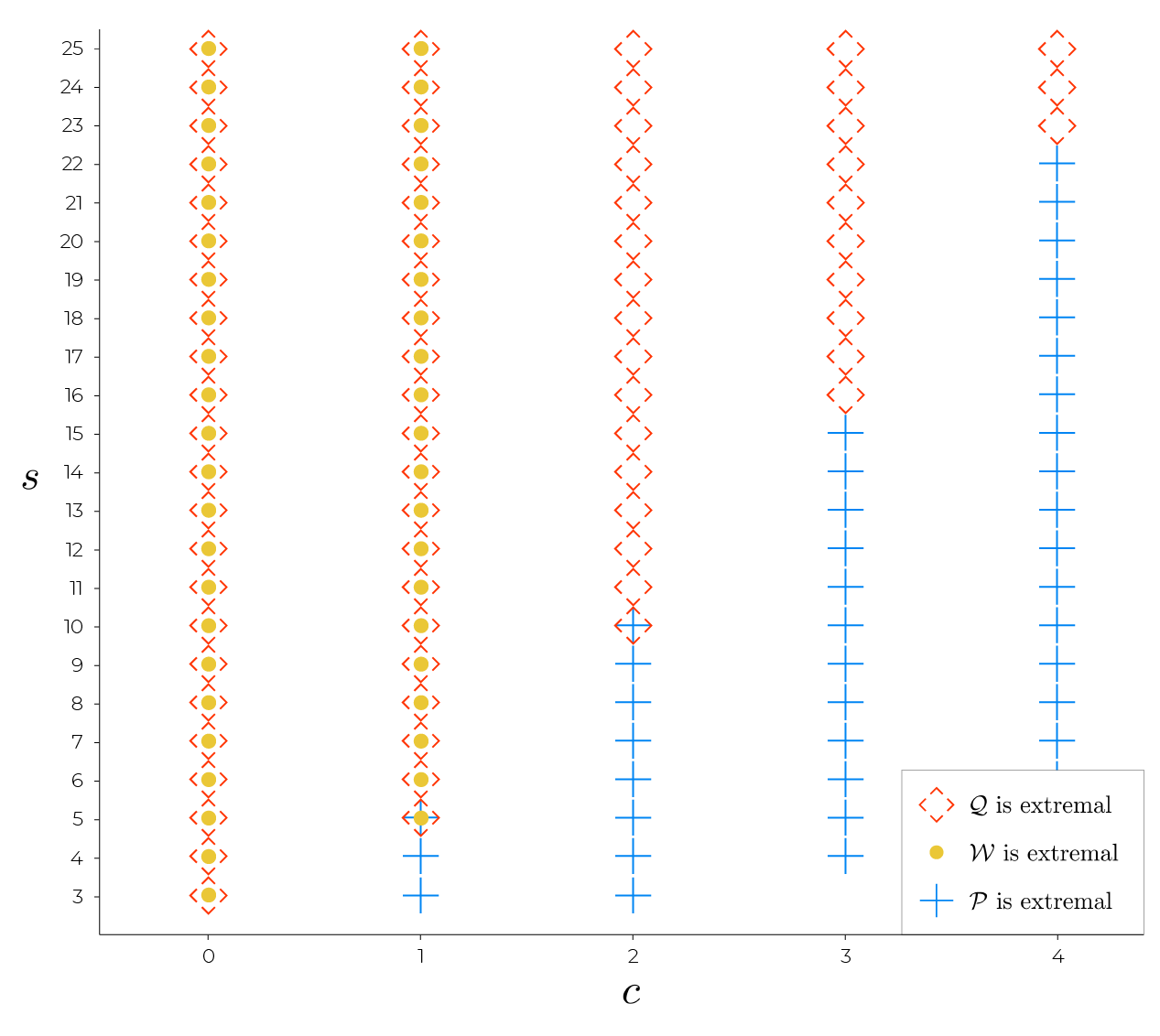}

\caption{Extremal families for small $c,s$. Note that for $c=1, s=5$ there are $3$ different extremal families. For $\ell = 1$, that is, $s = c + 1$, the family $\mathcal{P}'$ coincides with $\mathcal{P}$ (both are equal to ${n \choose \geq (m + 1)}$) and thus is formally extremal. When $\mathcal{P}' \neq \mathcal{P}$, for $c \leq 9$ $\mathcal{P}'$ cannot be extremal.}

\label{fig2}

\end{figure}
On Figures~\ref{fig1} and~\ref{fig2} we show, which families are extremal for different regimes of the parameters $s,c$. For some values we get that three different families are extremal at the same time.

We define shifted families in the next section. We should note that actually there are rather natural examples of families with no $s$-matching interpolating between $\mathcal{P}'(s, l)$ and $\mathcal{Q}(s, l)$ in a somewhat similar way as $\aaa_i$ interpolate between $\aaa_0$ and $\aaa_k$, but, as in the case of the EMC, there is a certain convexity that leads to the fact that it is the endpoints that must be extremal.


In the proof we will work only with sets of size $3$ or less. Therefore, any family, avoiding $s$-matching, must miss at least as many sets of size $3$ or less, as the extremal family. We thus get the following theorem about the truncated boolean lattice, confirming a conjecture of Frankl and the first author \cite{FK9} in our regime of the parameters. 

\begin{theorem} \label{t.truncated_lattice}
    Let $n,s,\ell,c$ be positive integers such that $n=2s+c=3s-\ell$, and $c,\ell\in [s-1]$. If $\mathcal{F} \subset {[n] \choose \leq 3}$ has no $s$-matching, then
    $$|\mathcal{F}| \leq \max\big\{|\mathcal{P}(s, \ell)^{(\leq 3)}|, |\mathcal{P}'(s, \ell)^{(\leq 3)}|, |\mathcal{Q}(s, \ell)^{(\leq 3)}|, |\mathcal{W}(s, \ell)^{(\leq 3)}|\big\}.$$
\end{theorem}

Note that a similar statement about $2$ first layers of boolean lattice is obviously false. Indeed, one of the families $\aaa_1^{(2)}(n,s-1), \aaa_2^{(2)}(n,s-1)$ has a larger cardinality than families $\mathcal{P}(s, \ell)^{(\leq 2)}, \mathcal{P}'(s, \ell)^{(\leq 2)}, \mathcal{Q}(s, \ell)^{(\leq 2)}, \mathcal{W}(s, \ell)^{(\leq 2)}$.

In Section~\ref{sec2}, we prove several easy facts and make the necessary preparations for the proof of the main theorem. In Section~\ref{sec3} we describe the strategy of the proof of the main theorem. 

\section{Preliminaries}\label{sec2}
Let us show that the families from Theorem~\ref{t.main} indeed have no $s$-matchings. We use the following simple lemma.

\begin{lemma} \label{l.weak_duality}
	Let $x_1, \ldots, x_n\ge 0$ be nonnegative real numbers such that $\sum_{i=1}^{n}x_i < s$. Suppose that for a family $\mathcal{F}$ and  any $F \in \mathcal{F}$ we have $\sum_{i\in F}x_i \geq 1$. Then $\nu(\mathcal{F}) < s$.
\end{lemma}

\begin{proof}
    Take any matching $F_1, \ldots, F_q\in \mathcal{F}$. Then $s>\sum_{i=1}^{n}x_i \geq \sum_{i\in \bigcup_{j} F_j}x_i = \sum_{j=1}^{q}\sum_{i\in F_j}x_i \geq q$. Thus, $q\le s-1$.
\end{proof}

\begin{corollary}
	The families $\mathcal{P}(s, \ell)$, $\mathcal{P'}(s, \ell)$, $\mathcal{Q}(s, \ell)$, $\mathcal{W}(s, \ell)$ do not contain an $s$-matching.
\end{corollary}

\tolerance=1000
\begin{proof}
	We show it by finding suitable $x_i$ and applying Lemma \ref{l.weak_duality}. \begin{itemize}
	    \item[] $\mathcal{P}(s, \ell)$: put $x_i = \frac{2}{3}$ for $i \leq \ell - 1$ and $x_i = \frac{1}{3}$ for $i \geq \ell$.
        \item[] $\mathcal{P}'(s, \ell)$: put $x_i = \frac{1}{2}$ for $i \leq 2\ell-1$ and $x_i = \frac{1}{3}$ for $i \geq 2\ell$.
        \item[] $\mathcal{Q}(s, \ell)$: put $x_i = \frac{1}{2}$ for $i \leq 2s - c - 1$ and $x_i = \frac{1}{4}$ for $i \geq 2s-c$.
        \item[] $\mathcal{W}(s, \ell)$: put $x_i = \frac{1}{2}$ for $i \leq 2s - 1$ and $x_i = 0$ for $i \geq 2s$.
	\end{itemize}   
\end{proof}

The families that we consider contain almost all subsets of $[n]$. Thus, it is more convenient to work with  the complements of the families.
For a family $\mathcal{F} \subset 2^{[n]}$, denote  by $\overline{\mathcal{F}}$ its complement: $\overline{\mathcal{F}}:=2^{[n]} \setminus \mathcal{F}$.

Simple computations show that 

\begin{align}
\label{eq_P_complement} |\overline{\mathcal{P}(s, \ell)}| &= {n-\ell+1 \choose 2} + n + 1 = {s+2c+1 \choose 2} + n + 1,\\
\label{eq_P'_complement} |\overline{\mathcal{P}'(s, \ell)}| &= {n \choose 2} - {2\ell-1 \choose 2} + n + 1 = (6c+4)s-\frac{3}{2}c^2-\frac{5}{2}c, \\
\label{eq_Q_complement} |\overline{\mathcal{Q}(s, \ell)}| &= (4c+4)s + \frac{4c^3-4c}{3}, \\
\label{eq_W'_complement} |\overline{\mathcal{W}(s, \ell)}| &= 2^{c+2}s.
\end{align}


For a family $\ff$, let $$\mathcal{F}^{(i)} = \mathcal{F} \cap {[n] \choose i}$$ be the $i$-th layer of $\mathcal{F}$. Let $y_{\mathcal{F}}(i)$ be the number of $i$-element sets, which are not in $\mathcal{F}$, that is, 
$$y_{\mathcal{F}}(i) = {n \choose i} - |\mathcal{F}^{(i)}|.$$ 
We will write $y(i)$ instead of $y_{\mathcal{F}}(i)$ if the family is clear from the context. 

Recall the definition of  \textit{shifting}. Given a  pair of indices $i<j\in [n]$ and a set $A \in 2^{[n]}$  define the $(i\leftarrow j)$-shift $S_{i\leftarrow j}(A)$ of $A$ as follows.
\[S_{i\leftarrow j}(A) := \begin{cases}A \ \ \ \ \ \ \ \ \ \ \ \ \ \ \ \ \ \ \ \ \text{if } i\in A\ \ \text{or }\ \ \ j\notin A;\\(A-\{j\})\cup \{i\}\ \ \text{if } i\notin A\ \ \text{and }\ j\in A.
\end{cases}\]
We say that a set $A$ {\it can be shifted to} $B$ if there is a sequence of $(i\leftarrow j)$-shifts with $i<j$ that transforms $A$ into $B$. For a family $\mathcal F\subset 2^{[n]}$, define the  $(i \leftarrow j)$-shift $S_{i\leftarrow j}(\mathcal F)$ as follows.
\[S_{i\leftarrow j}(\mathcal F) := \{S_{i,j}(A): A\in \mathcal F\}\cup \{A: A,S_{i\leftarrow j}(A)\in \mathcal F\}.\]
We call a family $\ff$ \textit{shifted}, if $S_{i\leftarrow  j}(\ff) = \ff$ for all $1\le i<j\le n$.\\

A family $\mathcal F$ is called an {\it up-set} if for any $F\in \mathcal F$ and $A\supset F$ we have $A\in \m F$.  For the proof of Theorem~\ref{t.main} we may restrict our attention to the families that are shifted up-sets (cf. e.g. \cite{F3} for a proof), which we assume for the rest of the paper. 

We will need the following simple, but very useful observation.

\begin{lemma} \label{l.EG}
	If $\mathcal{F} \subset 2^{[n]}$ is shifted and $\{i, j\} \notin \mathcal{F}, i < j$, then $y(2) \geq \frac{(n+j-2i)(n-j+1)}{2}$. Moreover, equality is achieved only if $\mathcal{F}^{(2)}$ is the family of all two-element sets that cannot be shifted to $\{i, j\}$.
\end{lemma}

\begin{proof}
	A set $\{a,b\}$ with $a<b$ can be shifted to $i,j$ iff $a\ge i$ and $b\ge j.$  It is a simple calculation to check that there are  $\frac{(n+j-2i)(n-j+1)}{2}$ two-element sets that can be shifted to $\{i, j\}$. All of them are not in $\mathcal{F}$ because $\ff$ is shifted. Equality is achieved only if all other two-element sets are in $\mathcal{F}$.
\end{proof}

The following theorem is the $2$-uniform case of the Erd\H os Matching Conjecture, proven by Erd\H{o}s and Gallai \cite{EG}. Later, Frankl in \cite{F3} gave a proof of this result that uses Lemma \ref{l.EG}. We will use the same method in the proof of our main theorem, and thus we present its proof. We also state it in a bit unusual form.

\begin{theorem} \label{t.Erdos_Gallai}
	Let $n,s$ be positive integers, $n\geq 2s$. Suppose that  $\mathcal{F} \subset 2^{[n]}$ satisfies $\nu(\mathcal{F}^{(2)}) < s$. Then $y(2) \geq \min({n - s + 1 \choose 2}, {n \choose 2} - {2s-1 \choose 2})$.
\end{theorem}
\begin{proof}
    As we have mentioned, we may assume that $\ff$ is shifted, cf.~\cite{F3}.
	Consider the sets $\{i, 2s+1-i\}$, $i\in [s]$. These sets form an $s$-matching, and thus at least one of them is not in $\mathcal{F}$. By Lemma \ref{l.EG}, we get $$y(2) \geq \frac{(n+2s+1-3i)(n-2s+i)}{2}$$ for some $i \in [1, s]$. The right-hand side is upward convex as a function of $i$. Thus, its minimum in the interval $[1, s]$ is attained at one of its ends. We get the bound 
    $$y(2) \geq \min\Big\{{n - s + 1 \choose 2}, \frac{(n+2s-2)(n-2s+1)}{2}\Big\}.$$ This is the claimed bound since ${n \choose 2} - {2s-1 \choose 2} = \frac{(n+2s-2)(n-2s+1)}{2}$.
\end{proof} 

\section{An overview of the proof}\label{sec3}

Let $\mathcal{F} \subset 2^{[n]}$ be a shifted up-set with no $s$-matching of maximal size. We need to prove that
$$\sum_{i=0}^{n} y_{F}(i) \geq \min\big\{|\overline{\mathcal{P}(s, \ell)}|, |\overline{\mathcal{P}'(s, \ell)}|, |\overline{\mathcal{Q}(s, \ell)}|, |\overline{\mathcal{W}(s, \ell)}|\big\},$$
and examine when equality is achieved.
As we have mentioned, we will actually prove the following stronger inequality.  $$\sum_{i=0}^{3} y_{F}(i) \geq \min\big\{|\overline{\mathcal{P}(s, \ell)}|, |\overline{\mathcal{P}'(s, \ell)}|, |\overline{\mathcal{Q}(s, \ell)}|, |\overline{\mathcal{W}(s, \ell)}|\big\}.$$ 

Note that this inequality implies Theorem \ref{t.truncated_lattice} as well as Theorem \ref{t.main}.


The first step in the proof is to show that $\m F$ does not contain the empty set or $1$-element sets. Actually, it is trivial for the empty set since its copies form a matching of an arbitrary size. For $1$-element sets, we use an inductive argument. From that point, we need to concentrate on $2$- and $3$-element sets only.\medskip

The balance between $2$- and $3$-element sets in $\m F$ is captured by the following key definition. The definition for even $d$ is more transparent, so the reader may ignore the case of odd $d$ for the moment and assume that it somehow interpolates between the even values of $d.$

\begin{defn}\label{def1} Denote by $d(\mathcal{F})$ the smallest $d \geq 0$ that satisfies one of the following two conditions:
\begin{itemize}
	\item $d$ is even and for some $i \in \big[1, \ell + \frac{d}{2}\big]$ the set $\{i, 2\ell+d+1-i\}$ is not in $\mathcal{F}$.
	\item $d$ is odd and either $\{1, 2\ell+d\} \notin \mathcal{F}$, or for some $i \in [3, \ell + \frac{d+1}{2}]$ the set $\{i, 2\ell+d+2-i\}$ is not in $\mathcal{F}$. 
\end{itemize}
\end{defn}

Since $\nu(\mathcal{F}) < s$, at least one of the sets $\{i, 2s+1-i\}$, $i \in [s]$, is not in $\mathcal{F}$. Thus,  $d=2c$ satisfies the condition from Definition~\ref{def1}. It means that $d(\mathcal{F})$ is well-defined and 
\begin{equation}\label{eqd2c}
d(\m F)\le 2c.
\end{equation} 

Note that if $d(\mathcal{F}) > 2k$ for some integer $k \geq 0$, then $\nu(\mathcal{F}^{(2)}) \geq \ell+k$. Indeed, the sets $\{1, 2\ell+2k\}, \ldots \{\ell+k, \ell+k+1\}$ form an $(\ell+k)$-matching and all belong to $\mathcal{F}$ by the definition of $d$. Actually, it is not difficult to prove that $\nu(\mathcal{F}^{(2)}) \geq \ell+k$ if and only if $d(\mathcal{F}) > 2k$.  The value $d(\mathcal{F})$ may thus be understood as a certain half-integral refinement of $\nu(\mathcal{F}^{(2)})$.\medskip

We will often use the following simple claim.

\begin{cla} \label{c.monotone_conditions_on_d}
    If $\mathcal{F}$ is a shifted family and $d(\mathcal{F}) \geq d$, then one of the following two conditions is satisfied:
    \begin{itemize}
    	\item $d$ is even and for some $i \in \big[1, \ell + \frac{d}{2}\big]$ the set $\{i, 2\ell+d+1-i\}$ is not in $\mathcal{F}$.
    	\item $d$ is odd and either $\{1, 2\ell+d\} \notin \mathcal{F}$, or for some $i \in [3, \ell + \frac{d+1}{2}]$ the set $\{i, 2\ell+d+2-i\}$ is not in $\mathcal{F}$. 
\end{itemize}
\end{cla}

\begin{proof}
    By the definition of $d(\mathcal{F})$ one of these conditions is satisfied for $d(\mathcal{F}) \leq d$. Thus, we need to check, that if the condition is satisfied for some $d'$, then it is satisfied for $d' + 1$. We consider separately the cases of odd and even $d'$.
    \begin{itemize}
    	\item Let $d'$ be an even number, satisfying the condition from the definition of $d(\mathcal{F})$, that is, for some $i \in \big[1, \ell + \frac{d'}{2}\big]$ the set $\{i, 2\ell+d'+1-i\}$ is not in $\mathcal{F}$. If $i = 1$, then by shiftedness of $\mathcal{F}$ we have $\{i, 2\ell+d'+2-i\} = \{1, 2\ell+(d'+1)\} \notin \mathcal{F}$ and thus the first part of the condition for odd $d'$ is satisfied. If $i \geq 2$, then, again by shiftedness, $\{i + 1, 2\ell+d'+2-i\} \notin \mathcal{F}$. Since $\{i + 1, 2\ell+d'+2-i\} = \{i', 2l+(d'+1)+2-i'\}$ for $i' = i + 1 \in [3, \ell+\frac{d' + 2}{2}]$, the second part of condition for odd $d'$ is satisfied.
    	\item Let $d'$ be an odd number, satisfying the condition from the definition of $d(\mathcal{F})$, that is, either $\{1, 2l+d'\} \notin \mathcal{F}$, or for some $i \in [3, \ell + \frac{d'+1}{2}]$ the set $\{i, 2\ell+d'+2-i\}$ is not in $\mathcal{F}$. If the first part of the condition is satisfied, then by shiftedness $\{1, 2\ell+(d'+1)\} \notin \mathcal{F}$, that is, the condition for $d' + 1$ is satisfied with $i=1$. If the second part of the condition is satisfied, then we notice, that the set $\{i, 2\ell+d'+2-i\}$ is exactly the set $\{i, 2\ell+(d'+1)+1-i\}$, so the condition for $d' + 1$ is satisfied with the same $i$ and the same missing set.
    \end{itemize}
\end{proof}

In Section \ref{s.y2} we prove lower bounds on $y_{\mathcal{F}}(2)$ in terms of $d(\mathcal{F})$. The bound is somewhat similar to the bound in Theorem \ref{t.Erdos_Gallai}. In particular, the bound will be a minimum of some expressions, and in order to use it, we will need to consider cases depending on where the minimum is achieved.

The second key ingredient of the proof is a certain averaging argument that allows us to bound $y_{\mathcal{F}}(3)$ in terms of $d(\mathcal{F})$.  The averaging will be over a certain class of matching with a rather tricky construction. This argument is where the definition of $d(\m F)$ for odd $d$ comes from. We comment on the subtlety in the definition of $d(\m F)$ after the statement of Lemma~\ref{l.even_d_to_y3}. In Section \ref{s.y3} we will use it to get several bounds on $y_{\mathcal{F}}(3)$ for different regimes of $d(\mathcal{F})$. Again, there will be several regimes to consider.

In Section~\ref{s.cases} we combine the bounds from Sections~\ref{s.y2} and~\ref{s.y3} and prove the main theorem for $c \geq 5$. The nature of the problem with its multiple extremal examples forces us to do a non-trivial case analysis. 

The proofs of some  technical statements are moved to Appendix A. In Appendix B, we provide the rather tedious proof of the main theorem for $c\le 4$. It uses the same ideas, but extra care is required to handle all the remaining cases.

\section{Avoiding $1$-element sets}
In order to get rid of $1$-element sets, we need to run an inductive procedure. We formalize it as follows. Assume that the following theorem holds.

\begin{thm}\label{t.relaxed}    Let $n,s,\ell$ be non-negative integers such that $n = 3s-\ell$ and $0\le \ell\le s$. Suppose that $\mathcal{F} \subset 2^{[n]}$ is a shifted up-set with $\nu(\mathcal{F}) < s$ and $\mathcal{F} \cap {[n] \choose 1} = \emptyset$. Then $|\mathcal{F}| \leq \max\big\{|\mathcal{P}(s, \ell)|, |\mathcal{P}'(s, \ell)|, |\mathcal{Q}(s, \ell)|, |\mathcal{W}(s, \ell)|\big\}$. Moreover, equality is achieved only if $\mathcal{F}$ coincides with one of these families.
\end{thm}
\begin{proof}[Proof of Theorem~\ref{t.main} using Theorem~\ref{t.relaxed}]
We prove Theorem~\ref{t.main} by induction on $\ell$. For $\ell \in \{0, 1\}$ the statement is given by Theorem~\ref{thmkl}, together with uniqueness that is proven in \cite{Kl}. In what follows, we assume that $\ell\ge 2.$


If $\{1\} \notin \mathcal{F}$ then by shiftedness $\mathcal{F} \cap {[n] \choose 1} = \emptyset$ and we are done. 
If $\{1\} \in \mathcal{F}$ then we apply induction. Let $\mathcal{F'} = \mathcal{F} \cap 2^{[2, n]}$. The family $\m F'$ is a family of subsets of an $(n-1)$-element set, $\nu(\mathcal{F'}) < s-1$, and we have $n-1 =3(s-1)-(\ell-2).$ Note that $\ell-2\ge 0$ and $\ell-2<s-1$, and thus we may apply induction to $\m F'$, getting that
$$|\overline{\mathcal{F'}}| \geq \min\big\{|\overline{\mathcal{P'}(s-1, \ell-2)}|, |\overline{\mathcal{P}(s-1, \ell-2)}|, |\overline{\mathcal{Q}(s-1, \ell-2)}|, |\overline{\mathcal{W}(s-1, \ell-2)}|\big\},$$
where the complement is taken with respect to $2^{[2,n]}.$ Clearly, $|\overline{\mathcal{F}}|\ge |\overline{\mathcal{F'}}|$, and so Theorem~\ref{t.main} is implied by the following claim.
\begin{restatable}{cla}{claimeleven}  We have 
    \begin{multline*}
        \min\big\{|\overline{\mathcal{P'}(s-1, \ell-2)}|, |\overline{\mathcal{P}(s-1, \ell-2)}|, |\overline{\mathcal{Q}(s-1, \ell-2)}|, \\|\overline{\mathcal{W}(s-1, \ell-2)}|\big\} >\min\big\{|\overline{\mathcal{P'}(s, \ell)}|, |\overline{\mathcal{P}(s, \ell)}|, |\overline{\mathcal{Q}(s, \ell)}|, |\overline{\mathcal{W}(s, \ell)}|\big\}.
    \end{multline*}
    Moreover,
    \begin{multline*}
        \min\big\{|\overline{\mathcal{P'}(s-1, \ell-2)}^{(\leq 3)}|, |\overline{\mathcal{P}(s-1, \ell-2)}^{(\leq 3)}|, |\overline{\mathcal{Q}(s-1, \ell-2)}^{(\leq 3)}|, \\|\overline{\mathcal{W}(s-1, \ell-2)}^{(\leq 3)}|\big\} >\min\big\{|\overline{\mathcal{P'}(s, \ell)}^{(\leq 3)}|, |\overline{\mathcal{P}(s, \ell)}^{(\leq 3)}|, |\overline{\mathcal{Q}(s, \ell)}^{(\leq 3)}|, |\overline{\mathcal{W}(s, \ell)}^{(\leq 3)}|\big\}.
    \end{multline*}
\label{c.no_siggletons}
\end{restatable}

(We need the moreover part for Theorem \ref{t.truncated_lattice}).

The proof the claim is a technical check, and we defer it to the appendix. This completes the proof of the implication. 
\end{proof}
In what follows, we need to prove Theorem~\ref{t.relaxed}. Concretely, we need to lower bound $y(2)+y(3).$

\section{Bounds on $y(2)$} \label{s.y2}

In this section, we prove a lower bound on $y_\mathcal{F}(2)$ that depends on $d(\mathcal{F})$. The definition of $d(\mathcal{F})$ depends on its parity, which is why we get slightly different bounds and proofs for odd and even $d(\mathcal{F})$. We start with the easier case of even $d$.

\begin{lemma} \label{l.even_d_to_y2}
	Let $d \geq 0$ be even. If $\mathcal{F} \subset 2^{[n]}$ is a shifted family with $d(\mathcal{F}) \leq d$ then {\small \begin{equation}\label{eqy21} y(2) \geq \min\left\{\frac{(4\ell+3c+d-2)(3c-d+1)}{2}, \frac{(\ell+3c-\frac{d}{2} + 1)(\ell+3c-\frac{d}{2})}{2}\right\}.\end{equation}}
    Moreover, equality is achieved only if $\mathcal{F}^{(2)} = {[2\ell+d-1] \choose 2}$ or $\mathcal{F}^{(2)} = \big\{F \in {n \choose 2 }: F\cap[\ell+\frac{d}{2}-1] \neq \emptyset\big\}$.
\end{lemma}

\begin{proof}
	The proof is very similar to Frankl's proof of the Erd\H{o}s-Gallai theorem.
    By Claim \ref{c.monotone_conditions_on_d} for some $i \in [\ell + \frac{d}{2}]$ the set $\{i, 2\ell+d+1-i\}$ is not in $\mathcal{F}$.
    By Lemma \ref{l.EG}, $$y(2) \geq \frac{(n+2\ell+d-3i+1)(n-2\ell-d+i)}{2} = \frac{(4\ell+3c+d-3i+1)(3c-d+i)}{2}.$$
	This bound is upward convex as a function of $i$, and thus the minimum on the interval $[1, \ell+\frac{d}{2}]$ is attained at one of its ends. Substituting $i=1,\frac{d}{2}$ in the inequality above, we get that \eqref{eqy21} holds. 
    Equality is attained only if $i\in\{1,\frac d2\}$ and $\ff^{(2)}$ consists of all sets that cannot be shifted to $\{i, 2\ell+d+1-i\}$. It implies that $\m F^{(2)}$ must have the form as stated in the moreover part. 
\end{proof}

We get the following claim as an immediate corollary.
\begin{cla}\label{clad0}
        Theorem~\ref{t.relaxed} holds for $d(\m F)=0$.
\end{cla}
\begin{proof}
   Indeed, Lemma \ref{l.even_d_to_y2} for $d(\m F)=0$ states that
{\small $$y_{\mathcal{F}}(2) \geq \min\left\{\frac{(4\ell+3c-2)(3c+1)}{2}, \frac{(\ell+3c+1)(\ell+3c)}{2}\right\} = \min\big\{y_{\mathcal{P}'(s,\ell)}(2), y_{\mathcal{P}(s,l)}(2)\big\}$$}
and equality is achieved only if $\mathcal{F}^{(2)} = \mathcal{P}'(s,\ell)^{(2)}$ or $\mathcal{F}^{(2)} = \mathcal{P}(s,\ell)^{(2)}$. At the same time, $\mathcal{F}$ does not contain $0$- or $1$-element sets and both $\mathcal{P}'(s,\ell)$ and $\mathcal{P}(s,\ell)$ contain all sets of size at least $3$. Altogether,  we get that $|\mathcal{F}| \leq \max\big\{|\mathcal{P}'(s, \ell)|, |\mathcal{P}(s, \ell)|\big\}$ and equality is achieved only if $\mathcal{F}\in\{\m P(s,\ell),\mathcal{P}'(s, \ell)\}$.
\end{proof}

The case of odd $d$ requires a more careful analysis. In this case, we  use the inequality \eqref{eqd2c} which states that $d(\mathcal{F}) \leq 2c$. 

\begin{lemma} \label{l.odd_d_to_y2}
	Let $d$ be a positive odd integer, $d \leq 2c$. If $\mathcal{F} \subset 2^{[n]}$ is a shifted family with $d(\mathcal{F}) = d$, then {\small \begin{equation}\label{eqy22}y(2) \geq \min\Big\{\frac{(4\ell+3c+d-2)(3c-d+1)}{2}, \frac{(\ell+3c-\frac{d-1}{2})(\ell+3c-\frac{d+1}2)}{2}\Big\}.\end{equation}} Moreover, equality is achieved only if $\mathcal{F}^{(2)} = {[2\ell+d-1] \choose 2}$ or $\mathcal{F}^{(2)} = \{F \in {n \choose 2 }: F\cap[\ell+\frac{d-1}{2}] \neq \emptyset\}$.
\end{lemma}

Combining the bounds \eqref{eqy21} and \eqref{eqy22}, we get the following bound.
{\small \begin{align}
	y(2) \geq \max\left\{\frac{(4\ell+3c+d-2)(3c-d+1)}{2}, \frac{(\ell+3c-\lceil \frac{d}{2} \rceil + 1)(\ell+3c-\lceil \frac{d}{2} \rceil)}{2}\right\}. \label{eq: universal d_to_y2}
\end{align}}

\begin{proof}[Proof of Lemma~\ref{l.odd_d_to_y2}]
	We consider cases depending on the missing set. If $\{1, 2\ell+d\} \notin \mathcal{F}$, then by Lemma \ref{l.EG} we have $y(2) \geq \frac{(4\ell+3c+d-2)(3c-d+1)}{2}$, and equality is achieved only if $\mathcal{F}^{(2)} = {[2\ell+d-1] \choose 2}$. Otherwise, one of the sets $\{i, 2\ell+d+2-i\}$ for some $i \in [3, \ell + \frac{d+1}{2}]$ is not in $\mathcal{F}$. By Lemma \ref{l.EG}, we get $$y(2) \geq \frac{(4\ell+3c+d-3i+2)(3c-d+i-1)}{2}=:f(i).$$ 
    Again, by convexity $y(2) \geq \min\big\{f(3), f\big(\ell + \frac{d+1}{2}\big)\big\}$. If $\ell + \frac{d+1}{2}=3$ then the two expressions coincide. If in that case $y(2) = f(3)$ then by Lemma \ref{l.EG}  $$\mathcal{F}^{(2)} = \Big\{F \in {n \choose 2 }: F\cap\Big[\ell+\frac{d-1}{2}\Big] \neq \emptyset\Big\}.$$ The family $\m F^{(2)}$ must have the same form if $y(2) = f\big(\ell + \frac{d+1}{2}\big)<f(3)$. 
    
    To complete the proof of the lemma, it is sufficient to show that either $f(3)>f\big(\ell + \frac{d+1}{2}\big)$ or $f(3) >\frac{(4\ell+3c+d-2)(3c-d+1)}{2}$ whenever  $\ell + \frac{d+1}{2}\ge 4$. This is shown in the following technical claim, which proof is deferred to the appendix. 
\begin{restatable}{cla}{claimfifteen}
        For $\ell + \frac{d+1}{2} \ge 4$ at least one of the inequalities
	\begin{align}
		\frac{(4\ell+3c+d-2)(3c-d+1)}{2} < \frac{(4\ell+3c+d-7)(3c-d+2)}{2} \label{eq: odd_d_to_y2: 1}
	\end{align}
	and
	\begin{align}
		\frac{(\ell+3c-\frac{d-1}{2})(\ell+3c-\frac{d+1}{2})}{2} < \frac{(4\ell+3c+d-7)(3c-d+2)}{2} \label{eq: odd_d_to_y2: 2}
	\end{align}
	is satisfied.
\end{restatable}
This completes the proof of the lemma.	
\end{proof}


Often it will be more convenient to use Lemmas~\ref{l.even_d_to_y2} and~\ref{l.odd_d_to_y2} in the following form. Again, the proof is a calculation that is deferred to the appendix.

\begin{restatable}{corollary}{corsixteen}
\label{c.y2}
	Let $\mathcal{F} \subset 2^{[n]}$ be a shifted up-set without $1$-element sets. Let $d(\mathcal{F}) = d \leq 2c$. Then either $$|\overline{\mathcal{P'}(s, \ell)}| - (y_{\mathcal{F}}(0) + y_{\mathcal{F}}(1) + y_{\mathcal{F}}(2)) \leq \frac{(4\ell+d-3)d}{2}$$ or $$|\overline{\mathcal{P}(s, \ell)}| - (y_{\mathcal{F}}(0) + y_{\mathcal{F}}(1) + y_{\mathcal{F}}(2)) \leq \frac{(2\ell+6c-\lceil\frac{d}{2}\rceil + 1)\lceil\frac{d}{2}\rceil}{2}.$$
\end{restatable}

\section{Bounds on $y(3)$} \label{s.y3}

In this section, we will prove several bounds on $y_{\mathcal{F}}(3)$ that depend on $d(\mathcal{F})$. We will use different bounds for different intervals of $d(\mathcal{F})$: small constant $d(\mathcal{F})$, $d(\mathcal{F}) \leq c + 1$ and $d(\mathcal{F}) \geq c + 2$. One reason why the bounds we use change near $d(\mathcal{F}) = c$ is that for one of the extremal families $\mathcal{F} = \mathcal{Q}(s, \ell)$ we have $d(\mathcal{F}) = c$.

Recall that we treated the case $d(\mathcal{F}) = 0$ in Claim~\ref{clad0}.

In what follows, $d(\m F)>0$. We will start with the bound that will be used for small $d(\mathcal{F})$ and for $\ell$ small compared to $c$.

\begin{lemma} \label{l.d_geq_1_to_y3}
	If $d(\mathcal{F}) > 0$ and $\nu(\mathcal{F}) < s$, then $y(3) \geq {3c-1 \choose 2}$. Moreover, equality is achieved only if $\mathcal{F}$ contains all sets in ${[n] \choose 3} \setminus {[2\ell+1, n] \choose 3}$.
\end{lemma}

\begin{proof}
	Recall that $s=\ell+c$. Since $d > 0$ the family $\mathcal{F}$ contains an $\ell$-matching $\{1, 2\ell\}, \ldots, \{\ell, \ell+1\}$. Consequently, $\nu(\mathcal{F} \cap 2^{[2\ell+1, n]}) < c$. At the same time, the set $[2\ell+1,n]$ has size $3c$. Consider any $c$-matching in ${[2\ell+1, n] \choose 3}$. It must contain at least one element in $\overline{\mathcal{F}}$. Averaging over all $c$-matchings in ${[2\ell+1, n] \choose 3}$, we get that $y_{\mathcal{F}}(3) \geq y_{\mathcal{F} \cap 2^{[2\ell+1, n]}}(3) \geq \frac{1}{c}{3c \choose 3} = {3c-1 \choose 2}$. The first inequality turns into an equality if and only if  $\mathcal{F}$ contains all sets in ${[n] \choose 3} \setminus {[2\ell+1, n] \choose 3}$.
\end{proof}

In the regime $d = d(\mathcal{F}) \leq c + 1$ we will need to consider two subcases. The easy case is when $\nu(\mathcal{F} \cap {[2\ell+d, n] \choose 3}) < c-d+1$. In that case, we will be able to bound $y(3)$ using missing sets in ${[2\ell+d,n]\choose 3}$.  Concretely wee make use of the following inequality from \cite{F3} that generalizes the Erd\H os--Ko--Rado theorem.

\begin{lemma}[\cite{F3}] \label{l.frankl_shifting}
	Let $\mathcal{G} \subset {[n] \choose k}, \nu(\mathcal{G}) < s, n \geq ks$. Then $|\mathcal{G}| \leq (s-1){n - 1 \choose k - 1}$.
\end{lemma}

For $d \geq 1$ we have $n-2\ell-d+1 = 3c-d+1 \geq 3(c-d+1)$. Thus, we can apply Lemma \ref{l.frankl_shifting} and get 
\begin{align*}y(3) &\geq {n-2\ell-d+1 \choose 3} - (c-d){n-2\ell-d \choose 2} \\
&= \Big(\frac{3c-d+1}{3}-(c-d)\Big){3c-d \choose 2} = \frac{(2d+1)(3c-d)(3c-d-1)}{6},\end{align*} provided  $\nu\big(\mathcal{F} \cap {[2\ell+d, n] \choose 3}\big) < c-d+1$.

Next, let us deal with the case $\nu(\mathcal{F} \cap {[2\ell+d, n] \choose 3}) \geq c-d+1$. This is one of the cruxes of the proof. The main idea of the bound for this case is as follows. Fix a $(c-d+1)$-matching in $\mathcal{F} \cap {[2\ell+d, n] \choose 3}$ and consider $s$-matchings that consist of three parts: \begin{itemize}
    \item[(i)] this fixed $(c-d+1)$-matching in $\mathcal{F} \cap {[2\ell+d, n] \choose 3}$;
    \item[(ii)] some $\ell$-matching in $\mathcal{F} \cap {[2\ell+d - 1] \choose 2}$; 
    \item[(iii)]  $d-1$ $3$-element sets, each containing one element in $[2\ell+d - 1]$ and two elements in $[2\ell+d, n]$.
\end{itemize} The sets from the first two groups belong to $\mathcal{F}$. Since $\nu(\mathcal{F}) < s$, at least one of the sets from this $s$-matching must belong to $\overline{\mathcal{F}}$, and it must be one of the $d-1$ sets from the third group. Averaging over all such matchings  gives us a lower bound on $y(3)$. Getting sufficiently many `missing' sets from the third group requires care. Implementation of these ideas depends on the parity of $d(\mathcal{F})$. We start with the easier case of odd $d$.

\begin{lemma} \label{l.odd_d_to_y3}
	Let $\mathcal{F} \subset 2^{[n]}$ be a shifted family with $\nu(\mathcal{F}) < s$. Suppose that $d(\mathcal{F}) \ge d$, where $d\leq c + 1$ and $d$ is odd. If $\nu\big(\mathcal{F} \cap {[2\ell + d, n] \choose 3}\big) \geq c-d+1$, then $$y(3) \geq (2\ell+d-1)(2d-3).$$
\end{lemma}

\begin{proof}
	Put $k = \frac{d-1}{2}$. By the definition of $d(\mathcal{F})$, the sets $\{1, 2\ell+2k\}, \ldots, \{\ell+k, \ell+k+1\}$ all belong to $\mathcal{F}$. They form an $(\ell+k)$-matching $\pi_L$ in $\mathcal{F} \cap {[2\ell+d-1] \choose 2}$.  Let $F_1, \ldots, F_{c-d+1}$ be a $(c-d+1)$-matching in $\mathcal{F} \cap {[2\ell+d,n] \choose 3},$ and put $$Y:=[2\ell+d,n]\setminus \bigcup_{i=1}^{c-d+1}F_i.$$ Note that $|Y| = (2\ell+3c) - (2\ell+d-1) - 3(c-d+1) = 2d-2 = 4k$. Let $\mathcal{G} := [2\ell+d-1]\times {Y \choose 2}$ be the family of three-element sets that contain one element from $[2\ell+d-1]$ and two elements from $Y$. We will prove that $$|\mathcal{G}\cap \overline{\mathcal{F}}| \geq (2\ell+2k)(4k-1) = (2\ell+d-1)(2d-3),$$ and thus that $y(3) \geq (2\ell+d-1)(2d-3)$.
	
	Construct a random $s$-matching $\pi$ in $(\mathcal{F} \cap {[2\ell+d-1] \choose 2}) \cup (\mathcal{F} \cap {[2\ell+d, n] \choose 3}) \cup \mathcal{G}$ by the following procedure.
	\begin{itemize}
		\item Add $F_1, \ldots, F_{c-d+1}$ to $\pi$.
		\item Choose $\ell$ sets from $\pi_{L}$ at random and add them to $\pi$. Let $Z$ be the set of $d-1$ elements in $[2l+d-1]$ that are not covered by these $\ell$ sets.
		\item So far we included $s-d+1$ sets from $\m F$ into $\pi$. We have $d-1$ uncovered elements in $Z$ and $2d-2$ uncovered elements in $Y$.
        \item Choose a random $(d-1)$-matching on $Z \cup Y$, consisting of $3$-element sets that contain one element from $Z$ and two elements from $Y$ and add it to $\pi$.
	\end{itemize}

	Let $\xi$ be a random variable equal to the number of sets from $\mathcal{G} \cap \overline{\mathcal{F}}$ in $\pi$. We use a double counting argument for $\E\xi$. If $\xi = 0$, then all $s$ sets from $\pi$ are in $\mathcal{F}$. It contradicts $\nu(\mathcal{F}) < s$. Thus, $\xi \geq 1$ and, consequently, $\E\xi \geq 1$.   
    
    Take a  set $\{x, y_1, y_2\} \in \mathcal{G}$ with $x\in [2\ell+d-1]$ and $y_1,y_2\in Y$. We have $\Pr[x \in Z] = \frac{k}{\ell+k}$ since, in order to be in $Z$, $x$ must be in one of the $k$ sets from $\pi_{L}$ that are not added to $\pi$. Next, we have $$\Pr[\{x, y_1, y_2\} \in \pi| x \in Z] = {|Y| \choose 2}^{-1} = {2d-2 \choose 2}^{-1},$$ since $x$ is matched with a random pair of elements from $Y$. Combining these, we get $$\Pr[\{x, y_1, y_2\} \in \pi] = \frac{k}{\ell+k}{2d-2 \choose 2}^{-1}.$$ 
    By linearity of expectation, we have $\E\xi = |\mathcal{G}\cap \overline{\mathcal{F}}|\cdot \frac{k}{l+k}{2d-2 \choose 2}^{-1}$. Combining it with $\E\xi \geq 1$, we get
	$$y(3) \geq |\mathcal{G} \cap \overline{\mathcal{F}}| \geq \frac{\ell+k}{k}{2d-2 \choose 2} = \frac{2\ell+d-1}{d-1}{2d-2 \choose 2} = $$ $$= (2\ell+d-1)(2d-3).$$
\end{proof}

The case of even $d$ is more complicated, and the bound on $y(3)$ is slightly weaker. The averaging argument in the coming lemma is the reason for the more complicated definition of $d(\m F)$ in the odd case.

\begin{lemma} \label{l.even_d_to_y3}
	Let $\mathcal{F} \subset 2^{[n]}$ be a family with $\nu(\mathcal{F}) < s$. Suppose that $d(\m F)\ge d$, and $d\in [2,c+1]$ is an even integer. Let $X \subset [2\ell+d-1]$ be a set such that for any $x \in X$ the family $\mathcal{F} \cap {[2\ell+d-1] \setminus \{x\} \choose 2}$ contains an $\big(\ell+\frac{d-2}{2}\big)$-matching. If $\nu(\mathcal{F} \cap {[2\ell + d, n] \choose 3}) \geq c-d+1$, then $$y(3) \geq |X|{2d-2 \choose 2}\Big(1 - \frac{(2\ell+d-2)(d-2)}{(d-2)|X|+2\ell}\Big).$$
\end{lemma}

Note, that if $|X|=2\ell+d-1$ then we get $y(3) \geq (2\ell+d-1)(2d-3)$ as in Lemma \ref{l.odd_d_to_y3}. However, we could only guarantee $|X|=2\ell+d-2$, as we will show in Lemma \ref{l.big_X}. Thus, we get a slightly worse bound. The size of $X$ that we could guarantee depends on the condition for odd $d$ in the definition of $d(\mathcal{F})$ and is the reason for the `exceptional' set $\{1,2\ell+d\}$. It is now appropriate to illustrate the subtlety in the definition of $d(\m F)$ in this case. On the one hand, if, by analogy with the case of even $d$, we chose a requirement that $\{i, 2\ell+d+1-i\} \notin \mathcal{F}$ for some $i \in [1, \ell + \frac{d-1}{2}]$,  we would only be able to guarantee $|X|\geq \frac{2\ell+d}{2}$, which is not good enough. If, on the other hand, we change $\{1, 2\ell+d\} \notin \mathcal{F}$ to $\{2, 2\ell+d\} \notin \mathcal{F}$, we could reach the bound $|X|=2\ell+d-1$.  But then the bound on $y(2)$ in Lemma \ref{l.odd_d_to_y2} becomes worse and ultimately insufficient for our purposes.

\begin{proof}
	Let $k = \frac{d-2}{2}$. For $x \in X$ let $\pi_{x}$ be an $(\ell+k)$-matching in $\mathcal{F} \cap {[2\ell+d-1] \setminus \{x\} \choose 2}$.	
	Let $F_1, \ldots, F_{c-d+1}$ be a $(c-d+1)$-matching in $\mathcal{F} \cap {[2\ell+d,n] \choose 3}$ and $Y=[2\ell+d,n]\setminus \bigcup_{i=1}^{c-d+1}F_i$. Note, that $|Y| = (2\ell+3c) - (2\ell+d-1) - 3(c-d+1) = 2d-2$. Let $\mathcal{G} := X\times {Y \choose 2}$ be a family of $3$-element sets, containing one element from $X$ and two elements from $Y$. We shall prove that 
    $$|\mathcal{G}\cap \overline{\mathcal{F}}| \geq |X|{2d-2 \choose 2}\Big(1 - \frac{(2\ell+d-2)(d-2)}{(d-2)|X|+2\ell}\Big).$$
	
	Construct a random $s$-matching $\pi$ in $(\mathcal{F} \cap {[2\ell+d-1] \choose 2}) \cup (\mathcal{F} \cap {[2\ell+d, n] \choose 3}) \cup \mathcal{G}$ by the following procedure.
	\begin{itemize}
		\item Add $F_1, \ldots, F_{c-d+1}$ to $\pi$.
		\item Choose $z \in X$ at random.
		\item Choose $\ell$ sets from $\pi_{z}$ at random and add them to $\pi$. Let $Z$ be the set of $d-1$ elements in $[2\ell+d-1]$ that are not covered by these $\ell$ sets. That is, $Z$ consists of $z$ and the union of $k$ random sets from $\pi_{z}$.
		\item So far, we included $s-d+1$ sets from $\m F$ into $\pi$. We have $d-1$ uncovered elements in $Z$ and $2d-2$ uncovered elements in $Y$.
        \item Choose a random $(d-1)$-matching on $Z \cup Y$, consisting of $3$-element sets that contain one element in $Z$ and two elements in $Y$, and add it to $\pi$. 
	\end{itemize}
	
	Let $\xi$ be a random variable equal to the number of sets from $\mathcal{G} \cap \mathcal{F}$ in $\pi$.\footnote{Note the discrepancy in the definition of $\xi$ with the previous lemma.} If $\xi = d-1$, then all $s$ sets from $\pi$ are in $\mathcal{F}$. It contradicts  $\nu(\mathcal{F}) < s$, and so $\xi \leq d-2$. Therefore, $\E\xi \leq d-2$. 
    
    Take a set $\{x, y_1, y_2\} \in \mathcal{G}$ with $x\in X$ and $y_1,y_2\in Y$. We have $\Pr[x \in Z] = \frac{1}{|X|} + \frac{|X|-1}{|X|}\frac{k}{k+\ell} = \frac{k|X|+\ell}{(\ell+k)|X|}$, where the first summand is the probability that $z=x$, and the second summand is the probability that $z\neq x$, but $x$ lies in one of the $k$ sets from $\pi_{z}$ that are not added in $\pi$. As in Lemma \ref{l.odd_d_to_y3}, $x$ is matched to a random pair of elements in $Y$, and we get 
    $$\Pr[\{x, y_1, y_2\} \in \pi| x \in Z] = {|Y| \choose 2}^{-1} = {2d-2 \choose 2}^{-1}.$$ Combining these, we have 
    $$P[\{x, y_1, y_2\} \in \pi] = \frac{k|X|+\ell}{(\ell+k)|X|}{2d-2 \choose 2}^{-1}$$ and $\E\xi = |\mathcal{G}\cap \mathcal{F}|\cdot \frac{k|X|+\ell}{(\ell+k)|X|}{2d-2 \choose 2}^{-1}$. Combining it with $\E\xi \leq d-2$, we get
	$$|\mathcal{G}\cap \mathcal{F}| \leq (d-2)\frac{(\ell+k)|X|}{k|X|+\ell}{2d-2 \choose 2}$$
	and
	$$y(3) \geq |\mathcal{G}| - |\mathcal{G}\cap \mathcal{F}| \geq |X|{2d-2 \choose 2} - (d-2)\frac{(\ell+k)|X|}{k|X|+\ell}{2d-2 \choose 2} = $$ $$ = |X|{2d-2 \choose 2} \left(1 - \frac{(d-2)(\ell+k)}{k|X|+\ell}\right) = |X|{2d-2 \choose 2} \left(1 - \frac{(d-2)(2\ell+d-2)}{(d-2)|X|+2l}\right).$$
	
\end{proof}

The next lemma guarantees that we can take $|X|\geq 2\ell+d-2$ when applying Lemma \ref{l.even_d_to_y3}.

\begin{lemma} \label{l.big_X}
	Let $\mathcal{F} \subset 2^{[n]}$ be a shifted family with $\nu(\mathcal{F}) < s$. Let $d = d(\mathcal{F})\ge 2$  be even. Then for any $x \in [2,2\ell+d-1]$ the family $\mathcal{F} \cap {[2\ell+d-1] \setminus \{x\} \choose 2}$ contains an $(\ell+\frac{d-2}{2})$-matching $\pi_x$.
\end{lemma}

\begin{proof}
	By the definition of $d(\mathcal{F}),$ the family $\mathcal{F} \cap {[2\ell+d-1] \setminus \{2\} \choose 2}$ contains an $(\ell+\frac{d-2}{2})$-matching $\pi_2 = \{\{1, 2\ell+d-1\}, \{3, 2\ell+d-2\}, \{4, 2\ell+d-3\}, \ldots, \{\ell+\frac{d}{2}, \ell + \frac{d}{2} + 1\}\}$. For $x \geq 2$ let $\{x, y\}$ be the set in $\pi_2$ that contains $x$. Since $\mathcal{F}$ is shifted, we get $\{2, y\} \in \mathcal{F}$. Let $\pi_x$ be a matching which is obtained from $\pi_2$ by replacing $\{x, y\}$ by $\{2, y\}$. Then $\pi_x$ is an $(\ell+\frac{d-2}{2})$-matching in $\mathcal{F} \cap {[2\ell+d-1] \setminus \{x\} \choose 2}$.
\end{proof}
Combining these two statements, we get the following.
\begin{corollary} \label{c.even_d_to_y3}
		Let $\mathcal{F} \subset 2^{[n]}$ be a shifted family with $\nu(\mathcal{F}) < s$. Let $d\in [2,c+1]$ be an even integer. If $d(\mathcal{F}) = d$, $\nu(\mathcal{F}) < s$ and $\nu(\mathcal{F} \cap {[2l + d, n] \choose 3}) \geq c-d+1$, then $$y(3) \geq (2\ell+d-2)(2d-3)\frac{1}{1+\frac{(d-2)^2}{2\ell(d-1)}}.$$
\end{corollary}

\begin{proof}
	Apply Lemma \ref{l.even_d_to_y3} with $X = [2,2l+d-1]$. Lemma \ref{l.big_X} guarantees that the condition on $X$ is satisfied. We get
	\begin{align*} y(3) &\geq (2\ell+d-2){2d-2 \choose 2}\Big(1 - \frac{(2\ell+d-2)(d-2)}{(d-2)(2\ell+d-2)+2\ell}\Big)\\
    &= (2\ell+d-2)(2d-3)\frac{1}{1+\frac{(d-2)^2}{2\ell(d-1)}}.\end{align*}
\end{proof}

Lemmas \ref{l.frankl_shifting}, \ref{l.odd_d_to_y3} and Corollary \ref{c.even_d_to_y3} together are enough to cover the range $5 \leq d(\mathcal{F}) \leq c + 1$, $\ell \geq \ell_0(c)$, where $\ell_0(c)$ is some linear function. 
In the case $d(\mathcal{F}) \geq c + 2$ we again distinguish cases depending on the relation between $\ell$ and $c$. For small $\ell=\ell(c)$ we also use Lemma \ref{l.d_geq_1_to_y3}. 
For large $\ell=\ell(c)$ we use the following  corollary of 
 Lemmas~\ref{l.odd_d_to_y3} and~\ref{l.even_d_to_y3}.
\begin{lemma} \label{l.d_geq_c_to_y3}
	Let $\mathcal{F} \subset 2^{[n]}$ be a shifted family with $\nu(\mathcal{F}) < s$ and $d(\mathcal{F}) \geq c + 2$. Then $y(3) \geq (2\ell+c)(2c-1)$.
\end{lemma}
\begin{proof}
    Assume that $c+1$ is odd. Then we could apply Lemma~\ref{l.odd_d_to_y3} with $d=c+1$ and get the inequality.  Next, assume that $c+1$ is even. Since $d(\m F)\ge c+2$, the family $\m F$ contains the matching $\{\{1, 2\ell+c+1\}, \ldots, \{\ell+\frac{c+1}{2}, \ell + \frac{c+3}{2}\}\}$. Removing the first set, we see that $\m F\cap{[2,2\ell+c]\choose 2}$ contains an $(\ell+c)$-matching. By shiftedness, $\m F\cap{[2\ell+c]\setminus \{x\}\choose 2}$ contains an $(\ell+c)$-matching for any $x\in [2\ell+c]$. We conclude that the conditions of Lemma~\ref{l.even_d_to_y3} hold with $d=c+1$ and $X= [2\ell+c]=[2\ell+d-1].$ Therefore, we conclude that $y(3)\ge (2\ell+d-1)(2d-3)=(2\ell+c)(2c-1)$ in this case as well.
\end{proof}

\section{Proof of Theorem~\ref{t.relaxed} for $c \geq 5$} \label{s.cases}

In this section, we combine the bounds from Sections \ref{s.y2} and \ref{s.y3} and prove Theorem \ref{t.relaxed} for $c \geq 5$.  Also for $c=3$ and $c=4$ we deal with all cases, except $d \in \{3, 4\}$.

We record the values $y(2) + y(3)$ for the families $\m P(s,\ell),\m P'(s,\ell)$ $\m Q(s,\ell)$ that may be obtained by straightforward calculations:
\begin{align}\label{yp} y_{\mathcal{P}(s, \ell)}(2) + y_{\mathcal{P}(s, \ell)}(3) &= {\ell + 3c + 1 \choose 2},\\
\label{yp'}y_{\mathcal{P}'(s, \ell)}(2) + y_{\mathcal{P}'(s, \ell)}(3) &= \frac{(4\ell+3c-2)(3c+1)}{2},\\
\label{yq} y_{\mathcal{Q}(s, \ell)}(2) + y_{\mathcal{Q}(s, \ell)}(3) &= (4c+2)\ell + \frac{4c^3+12c^2-c-3}{3}.\end{align}
Recall that \eqref{eq: universal d_to_y2} holds for $y_\ff(2)$. As a simplifying step, let us bound the following expression: $\min\{y_{\mathcal{P}(s, \ell)}(2) + y_{\mathcal{P}(s, \ell)}(3), y_{\mathcal{P}'(s, \ell)}(2) + y_{\mathcal{P}'(s, \ell)}(3)\}-y_{\ff}(2)$. If the minimum in \eqref{eq: universal d_to_y2} is attained on the first expression, then we subtract it from  $y_{\mathcal{P}'(s, \ell)}(2) + y_{\mathcal{P}'(s, \ell)}(3)$. If the minimum in \eqref{eq: universal d_to_y2} is attained on the second expression, then we subtract it from  $y_{\mathcal{P}(s, \ell)}(2) + y_{\mathcal{P}(s, \ell)}(3)$. As a result, skipping some calculations, we get that if the minimum in \eqref{eq: universal d_to_y2} is attained on the first expression, then
$$y_{\mathcal{P}'(s, \ell)}(2) + y_{\mathcal{P}'(s, \ell)}(3)-y_{\ff}(2)\le \frac{d(4\ell+d-3)}2 $$
and it is sufficient to show
\begin{align}\label{eq_desired_y3_P'}
y_{\mathcal{F}}(3) > \frac{d(4\ell+d-3)}2
\end{align}
and if the minimum in \eqref{eq: universal d_to_y2} is attained on the second expression, then
$$y_{\mathcal{P}(s, \ell)}(2) + y_{\mathcal{P}(s, \ell)}(3)-y_{\ff}(2)\le \frac{\Big\lceil\frac d2\Big\rceil\left(2\ell+6c-\big\lceil\frac d2\big\rceil + 1\right)}{2} $$
and it is sufficient to show
\begin{align}\label{eq_desired_y3_P}
y_{\mathcal{F}}(3) > \frac{\Big\lceil\frac d2\Big\rceil\left(2\ell+6c-\big\lceil\frac d2\big\rceil + 1\right)}{2}.
\end{align}

In the cases when this does not hold, we will show that
\begin{align}\label{eq_desired_y3_Q}
y_\ff(2)+y_\ff(3)\le y_{\mathcal{Q}(s, \ell)}(2) + y_{\mathcal{Q}(s, \ell)}(3).
\end{align}
and equality is only possible if $\ff = \m Q(s,\ell)$. 

\subsection{$d\in [4]$}

We consider cases $d \in \{1, 2\}$ and $d \in \{3, 4\}$ separately.

For $d(\mathcal{F}) \in \{1, 2\}$ we use Lemma \ref{l.d_geq_1_to_y3} to bound $y_{\mathcal{F}}(3)$ and \eqref{eq: universal d_to_y2} to bound $y_{\mathcal{F}}(2)$. (We group these two cases since we have the same bound on $y_{\mathcal{F}}(3)$ for $d=1,2$.) 
Then we consider cases, depending on whether the minimum in \eqref{eq: universal d_to_y2} is attained on the first or on the second argument, as well as depending on whether $\ell$ is big or small compared to $c$. For each of these cases, we prove one of \eqref{eq_desired_y3_P}, \eqref{eq_desired_y3_P'}, \eqref{eq_desired_y3_Q} with strict inequality.

\begin{lemma} \label{l.d_1_2}
	Let $\mathcal{F}$ be a shifted family without one-element sets with $\nu(\mathcal{F}) < s, d(\mathcal{F}) = d$. If $d \in \{1, 2\}$ and $c \geq 3$, then $|\mathcal{F}| < \max\{|\mathcal{P}(s, \ell)|, |\mathcal{P'}(s, \ell)|, |\mathcal{Q}(s, \ell)|\}$.
\end{lemma}

Note that in this lemma we require only $c \geq 3$ instead of  $c \geq 5$, and thus we will be able to use this lemma when dealing with small cases in the appendix.

\begin{proof}
    By \eqref{eq: universal d_to_y2} we get
    $$y_{\mathcal{F}}(2)\geq \min(\frac{(4\ell+3c)(3c-1)}{2}, \frac{(\ell+3c)(\ell+3c-1)}{2})$$
    and by Lemma \ref{l.d_geq_1_to_y3} we get
    $$y_{\mathcal{F}}(3) \geq  {3c-1 \choose 2}.$$
	If the minimum in \eqref{eq: universal d_to_y2} is attained on the first argument and $\ell < \frac{9c^2 - 9c + 4}{8}$, then we prove \eqref{eq_desired_y3_P'}.
	$$y_{\mathcal{F}}(3) \geq  {3c-1 \choose 2} > (4\ell-1) 
    \geq \frac{d(4\ell+d-3)}{2}.$$
	If the minimum in \eqref{eq: universal d_to_y2} is attained on the first argument and $\ell > \frac{4c^2 - 7c + 3}{6}$, then, using assumption $c > 2$, below we show the validity of \eqref{eq_desired_y3_Q} with strict inequality. (We skip some simplifying calculations in the first inequality below.) 
	
	\[
	\begin{split}
		&y_{\mathcal{F}}(2) + y_{\mathcal{F}}(3) \geq (y_{\mathcal{Q}(s, \ell)}(2) + y_{\mathcal{Q}(s, \ell)}(3)) + (2c-4)\ell - \frac{4c^3-15c^2+17c-6}{3} = \\& = (y_{\mathcal{Q}(s, \ell)}(2) + y_{\mathcal{Q}(s, \ell)}(3)) + (2c-4)(\ell - \frac{4c^2-7c+3}{6}) > y_{\mathcal{Q}(s, \ell)}(2) + y_{\mathcal{Q}(s, \ell)}(3).
	\end{split}
	\]
	
	Since $\frac{4c^2-7c+3}{6} < \frac{9c^2-9c+4}{8}$ for $c \geq 0$, at least one of the inequalities $\ell < \frac{9c^2 - 9c + 4}{8}$ and $\ell > \frac{4c^2 - 7c + 3}{6}$ is satisfied. Therefore, if the minimum in \eqref{eq: universal d_to_y2} is attained in the first argument, then $|\mathcal{F}| > \min\{|\mathcal{P}'(s, \ell)|, |\mathcal{Q}(s, \ell)|\}$.
	
	If the minimum in \eqref{eq: universal d_to_y2} is attained on the second argument and $\ell < \frac{9c^2-15c+2}{2}$, then we prove \eqref{eq_desired_y3_P}.
	$$y_{\mathcal{F}}(3) \geq {3c - 1 \choose 2} > (\ell + 3c) = \frac{\Big\lceil\frac d2\Big\rceil\left(2\ell+6c-\big\lceil\frac d2\big\rceil + 1\right)}{2}.$$
	
	If the minimum in \eqref{eq: universal d_to_y2} is attained on the second argument, $\ell \geq \frac{9c^2-15c+2}{2}$ and $\ell > \frac{4(4c^3-15c^2+17c-6)}{3(9c^2-19c-8)}$, then we get \eqref{eq_desired_y3_Q} with strict inequality.
	{\small
    \[
	\begin{split}
		&y_{\mathcal{F}}(2) + y_{\mathcal{F}}(3) \geq (y_{\mathcal{Q}(s, \ell)}(2) + y_{\mathcal{Q}(s, \ell)}(3)) + \left(\frac{\ell}{2}-\frac{2c+5}{2}\right)\ell - \frac{4c^3-15c^2+17c-6}{3}  \\& \geq (y_{\mathcal{Q}(s, \ell)}(2) + y_{\mathcal{Q}(s, \ell)}(3)) + \left(\frac{9c^2-15c+2}{4}-\frac{2c+5}{2}\right)\ell - \frac{4c^3-15c^2+17c-6}{3} \\& = (y_{\mathcal{Q}(s, \ell)}(2) + y_{\mathcal{Q}(s, \ell)}(3)) + \frac{9c^2-19c-8}{4}\ell - \frac{4c^3-15c^2+17c-6}{3}\\
        &> y_{\mathcal{Q}(s, \ell)}(2) + y_{\mathcal{Q}(s, \ell)}(3).
	\end{split}
	\]
	}
    
	Since for $c \geq 3$ we have $\frac{9c^2-15c+2}{2} > \frac{4(4c^3-15c^2+17c-6)}{3(9c^2-19c-8)}$, one of the conditions $\ell < \frac{9c^2-15c+2}{2}$ or ($\ell \geq \frac{9c^2-15c+2}{2}$ and $\ell > \frac{4(4c^3-15c^2+17c-6)}{3(9c^2-19c-8)}$) is satisfied. Therefore, if the minimum is attained on the second argument, then $|\mathcal{F}| > \min\{|\mathcal{P}(s, \ell)|, |\mathcal{Q}(s, \ell)|\}$.
\end{proof}

The case $d(\mathcal{F}) \in \{3, 4\}$ is intermediate between the cases $d(\mathcal{F}) \in \{1, 2\}$ and $d(\mathcal{F}) \geq 5$. If minimum in Lemma \ref{l.even_d_to_y2} is attained in the first argument, we argue as in the case $d(\mathcal{F}) \geq 5$ and if it is attained in the second argument, we argue as in the case $d(\mathcal{F}) \in \{1, 2\}$.

\begin{lemma} \label{l.d_3_4}
	Let $\mathcal{F}$ be a shifted family without one-element sets with $\nu(\mathcal{F}) < s, d(\mathcal{F}) = d$. If $d \in \{3, 4\}$ and $c \geq 5$, then $|\mathcal{F}| < \max\{|\mathcal{P}(s, \ell)|, |\mathcal{P'}(s, \ell)|, |\mathcal{Q}(s, \ell)|\}$.
\end{lemma}

\begin{proof}
	By \eqref{eq: universal d_to_y2} we get
	$$y_{\mathcal{F}}(2) \geq \min\Big\{\frac{(4\ell+3c+d-2)(3c-d+1)}{2}, \frac{(\ell + 3c-\lceil\frac{d}{2}\rceil + 1)(\ell + 3c-\lceil\frac{d}{2}\rceil)}{2}\Big\}.$$
	
	If the minimum in \eqref{eq: universal d_to_y2} is attained in the second argument then we argue as in Lemma \ref{l.d_1_2}. We use Lemma \ref{l.d_geq_1_to_y3} to bound $y(3)$.
	
	If $\ell < \frac{9c^2-21c+4}{4}$, then we establish \eqref{eq_desired_y3_P}.
	$$y_{\mathcal{F}}(3) \geq {3c - 1 \choose 2} > (2\ell + 6c-1) \geq \frac{\Big\lceil\frac d2\Big\rceil\left(2\ell+6c-\big\lceil\frac d2\big\rceil + 1\right)}{2}.$$

    If $\ell \geq \frac{9c^2-21c+4}{4}$ and $\ell > \frac{8(4c^3-15c^2+26c-9)}{3(9c^2-29c-24)}$, then
	{\small
    \[
	\begin{split}
		&y_{\mathcal{F}}(2) + y_{\mathcal{F}}(3) \geq 	(y_{\mathcal{Q}(s, \ell)}(2) + y_{\mathcal{Q}(s, \ell)}(3)) + \left(\frac{\ell}{2}-\frac{2c+7}{2}\right)\ell - \frac{4c^3-15c^2+26c-9}{3} \\& \geq (y_{\mathcal{Q}(s, \ell)}(2) + y_{\mathcal{Q}(s, \ell)}(3)) + \left(\frac{9c^2-21c+4}{8}-\frac{2c+7}{2}\right)\ell - \frac{4c^3-15c^2+26c-9}{3} \\& = (y_{\mathcal{Q}(s, \ell)}(2) + y_{\mathcal{Q}(s, \ell)}(3)) + \frac{9c^2-29c-24}{8}\ell - \frac{4c^3-15c^2+26c-9}{3} \\
        &> y_{\mathcal{Q}(s, \ell)}(2) + y_{\mathcal{Q}(s, \ell)}(3).
	\end{split}
	\]
    }


	Since for $c \geq 5$ we have $\frac{9c^2-21c+4}{4} > \frac{8(4c^3-15c^2+26c-9)}{3(9c^2-29c-24)}$, one of conditions $\ell < \frac{9c^2-21c+4}{4}$ or ($\ell \geq \frac{9c^2-21c+4}{4}$ and $\ell > \frac{8(4c^3-15c^2+26c-9)}{3(9c^2-29c-24)}$) is satisfied. Therefore, if the minimum is attained in the second argument, then $|\mathcal{F}| > \min(|\mathcal{P}(s, \ell)|, |\mathcal{Q}(s, \ell)|)$.


	
	
	Below, we deal with the case when the minimum in \eqref{eq: universal d_to_y2} is attained on the first argument. Here we need a more complicated argument. The same argument will be used in the next subsection for moderate $d$. We distinguish two more cases: whether $\nu(\mathcal{F} \cap {[2\ell+d, n] \choose 3}) \geq c - d + 1$ or $\nu(\mathcal{F} \cap {[2\ell+d, n] \choose 3}) < c - d + 1$. If $\nu(\mathcal{F} \cap {[2\ell+d, n] \choose 3}) \geq c - d + 1$, then we use Lemma~\ref{l.odd_d_to_y3} for $d=3$ and Lemma~\ref{l.even_d_to_y3} for $d=4$. If $\nu(\mathcal{F} \cap {[2\ell+d, n] \choose 3}) < c - d + 1$, then we use Lemma~\ref{l.frankl_shifting}.
	
	If $d = 3$ and $\nu(\mathcal{F} \cap {[2\ell+d, n] \choose 3}) \geq c - d + 1$, then by Lemma \ref{l.odd_d_to_y3} we get
    $$y(3) \geq 3(2\ell+2) > 6\ell = \frac{d(4\ell+d-3)}{2}$$
    and thus \eqref{eq_desired_y3_P'} is proven.

    If $d = 4$ and $\nu(\mathcal{F} \cap {[2\ell+d, n] \choose 3}) \geq c - d + 1$, then by  Corollary \ref{c.even_d_to_y3} we get
	$$y(3) \geq 5(2\ell+2) \cdot \frac{1}{1+\frac{2}{3\ell}}.$$
    Thus, to get \eqref{eq_desired_y3_P'} it is sufficient to check
    \begin{align}\label{eq_1}
        5(2\ell+2) \cdot \frac{1}{1+\frac{2}{3\ell}} > 8\ell+2.
    \end{align}
    Indeed, for $\ell \geq 3$ we get 
    $$5(2\ell+2) \cdot \frac{1}{1+\frac{2}{3\ell}} \geq \frac{90}{11}(\ell+1) > 8\ell+2.$$
    For $\ell = 2$ the left-hand side of \eqref{eq_1} equals $\frac{45}{2}$ and the right-hand side of \eqref{eq_1} equals $18$. For $\ell = 3$ the left-hand side \eqref{eq_1} equals $12$ and the right-hand side of \eqref{eq_1} equals $10$.
	
	Finally, we need to deal with the case $y_{\mathcal{F}}(2) \geq \frac{(4\ell+3c+d-2)(3c-d+1)}{2}$ and $\nu(\mathcal{F} \cap {[2\ell+d, n] \choose 3}) < c - d + 1$. Note, that in this case we use the assumption $d \in \{3, 4\}$ only to prove that the desired inequality is strict. 
	Since $\nu(\mathcal{F} \cap {[2\ell+d,n] \choose 3}) \geq 0$, the condition $\nu(\mathcal{F} \cap {[2\ell+d,n] \choose 3}) < c - d + 1$ implies $d \leq c$.

	Let $\mathcal{F}' = \mathcal{F} \cap {[2\ell+d, n] \choose 3}$. Since $\nu(\mathcal{F}') \leq c - d$, by Lemma \ref{l.frankl_shifting} we have $|\mathcal{F}'| \leq (c - d){3c - d \choose 2}$. Therefore, $$y_{\mathcal{F}}(3) \geq y_{\mathcal{F}'}(3) \geq {3c-d+1 \choose 3} - (c - d){3c - d \choose 2} = \frac{(2d+1)(3c-d)(3c-d-1)}{2}.$$
	Combining bounds on $y_{\mathcal{F}}(2)$ and $y_{\mathcal{F}}(3)$, we get $y_{\mathcal{F}}(2) + y_{\mathcal{F}}(3) \geq f_{\ell,c}(d)$, where $$f_{\ell,c}(d) := \frac{(2d+1)(3c-d)(3c-d-1)}{2} + \frac{(4\ell+3c+d-2)(3c-d+1)}{2}.$$
	Since $f_{\ell,c}''(d) = 6d - 12c + 2 < 0$ for $0 \leq d \leq c$ and $c \geq 1$, we get $y_{\mathcal{F}}(2) + y_{\mathcal{F}}(3) \geq \min(f_{\ell,c}(0), f_{\ell,c}(c))$. Since $f_{\ell,c}(0) = y_{\mathcal{P}'(s, \ell)}(2) + y_{\mathcal{P}'(s, \ell)}(3) + \frac{3c(3c-1)}{2} > y_{\mathcal{P}'(s, \ell)}(2) + y_{\mathcal{P}'(s, \ell)}(3)$ and $f_{\ell,c}(c) = y_{\mathcal{Q}(s, \ell)}(2) + y_{\mathcal{Q}(s, \ell)}(3)$, we get $y_{\mathcal{F}}(2) + y_{\mathcal{F}}(3) \geq \min\big\{ y_{\mathcal{P}'(s, \ell)}(2) + y_{\mathcal{P}'(s, \ell)}(3), y_{\mathcal{Q}(s, \ell)}(2) + y_{\mathcal{Q}(s, \ell)}(3)\big\}$ and equality holds only if $d=c$. Moreover, $\mathcal{F}$ contains no sets of size $\le 2$ and both $\mathcal{P}'(s, \ell)$ and $\mathcal{Q}(s, \ell)$ contain all sets of size  $\ge 4$. Therefore, $|\mathcal{F}| \leq \max\{|\mathcal{P}'(s, \ell)|, |\mathcal{Q}(s, \ell)|\}$. If equality holds, then $d = c$. Since it contradicts the assumption $d \in \{3, 4\}$ and $c \geq 5$, we get that $|\mathcal{F}| < \max(|\mathcal{P}'(s, \ell)|, |\mathcal{Q}(s, \ell)|)$.
\end{proof}

\subsection{Moderate $d$}

In this subsection we will deal with the case of moderate $d(\mathcal{F})$, that is, the case $5 \leq d(\mathcal{F}) \leq c + 1$. We consider two subcases: $\nu(\mathcal{F} \cap {[2\ell+d,n] \choose 3}) \geq c - d + 1$ and $\nu(\mathcal{F} \cap {[2\ell+d,n] \choose 3}) < c - d + 1$. In the first case, in order to bound $y_{\mathcal{F}}(3)$ we use Lemmas \ref{l.even_d_to_y3} and \ref{l.odd_d_to_y3} or Lemma \ref{l.d_geq_1_to_y3}, depending on the relation between $l$ and $c$. In the second case, we use Lemma \ref{l.frankl_shifting} to bound $y_{\mathcal{F}}(3)$.

\begin{lemma} \label{l.moderate_odd_d}
	Let $d(\mathcal{F}) = d$ be odd integer, $5 \leq d \leq c + 1$. If $\mathcal{F}$ is a family without singletons such that $d(\mathcal{F}) = d, \nu(\mathcal{F}) < s$ and $\nu(\mathcal{F} \cap {[2\ell+d,n] \choose 3}) \geq c - d + 1$, then $|F| < \max\{|\mathcal{P}(s, \ell)|, |\mathcal{P}'(s, \ell)|\}$.
\end{lemma}

\begin{proof}
    As we mentioned in the beginning of this section, it is sufficient to prove \eqref{eq_desired_y3_P'} and \eqref{eq_desired_y3_P}.
    Inequality \eqref{eq_desired_y3_P'} is an immediate corollary of Lemma \ref{l.odd_d_to_y3}. Indeed, by Lemma \ref{l.odd_d_to_y3} we have
	$$y(3) \geq (2\ell+d-1)(2d-3) = \frac{(4\ell+2d-2)(2d-3)}{2} > \frac{(4\ell+d-3)d}{2}.$$
	To prove \eqref{eq_desired_y3_P} we consider two cases: $\ell \leq 6c-27$ and $\ell \geq 6c-26$.
	If $\ell \leq 6c-27$, we use Lemma \ref{l.d_geq_1_to_y3} and get
	$$y(3) \geq {3c-1 \choose 2} = \frac{(2(6c-27+\frac{56}{c+2})+6c)\frac{c+2}{2}}{2} > \frac{(2\ell+6c-\frac{d}{2}+\frac{1}{2})(\frac{d+1}{2})}{2}.$$
	If $\ell \geq 6c-26$, we use Lemma \ref{l.odd_d_to_y3} and for $d \geq 5, \ell \geq 1$ get

	
    \[
	\begin{split}
		y(3) &\geq (2\ell+d-1)(2d-3)\\
        &= \frac{(3\ell+26-\frac{d}{2}+\frac{1}{2})\frac{d+1}{2}}{2} + \left(\frac{13}{4}d - \frac{27}{4}\right)\ell + \frac{17}{8}d^2 - \frac{23}{2}d - \frac{29}{8}  \\&
		\geq \frac{(3\ell+26-\frac{d}{2}+\frac{1}{2})\frac{d+1}{2}}{2} + \frac{17}{8}d^2 - \frac{33}{4}d - \frac{83}{8} > \frac{(3\ell+26-\frac{d}{2}+\frac{1}{2})\frac{d+1}{2}}{2} \\& \geq \frac{(2\ell+6c-\frac{d}{2}+\frac{1}{2})\frac{d+1}{2}}{2}.
	\end{split}
	\]
In the second inequality we used $\ell \geq 1$ and $\frac{13}{4}d - \frac{27}{4} > 0$. In the fourth inequality we used $d \geq 5$.	
\end{proof}

\begin{restatable}{lemma}{lemmatwseven} 
\label{l.moderate_even_d}
	Let $d(\mathcal{F}) = d$ be even integer, $6 \leq d \leq c + 1$. If $\mathcal{F}$ is a shifted family without singletons such that $d(\mathcal{F}) = d, \nu(\mathcal{F}) < s$ and $\nu(\mathcal{F} \cap {[2\ell+d,n] \choose 3}) \geq c - d + 1$, then $|F| < \max\{|\mathcal{P}(s, \ell)|, |\mathcal{P}'(s, \ell)|\}$.
\end{restatable}

The proof of Lemma \ref{l.moderate_even_d} is very similar to the proof of Lemma \ref{l.moderate_odd_d}. The only difference is that we need to consider the cases of small and large $l$ both for proving \eqref{eq_desired_y3_P'} and \eqref{eq_desired_y3_P}, because Corollary \ref{c.even_d_to_y3} provides us with a slightly worse bound than Lemma \ref{l.odd_d_to_y3}. We deferred the proof to the appendix.

Next, we deal with the case $\nu(\mathcal{F} \cap {[2\ell+d,n] \choose 3}) < c - d + 1$. We use almost the same argument as in the end of the proof of Lemma \ref{l.d_3_4}. The main difference is that now we do not know whether minimum in \eqref{eq: universal d_to_y2} is achieved on the first or on the second argument. Thus, we need to consider these two cases separately. 

\begin{lemma} \label{l.big_nu}
	Let $d(\mathcal{F}) = d$ be an integer, $d \geq 5$. If $\mathcal{F}$ is a shifted family without one-element sets such that $d(\mathcal{F}) = d$ and $\nu(\mathcal{F} \cap {[2\ell+d,n] \choose 3}) < c - d + 1$, then $|\mathcal{F}| \leq \max\{|\mathcal{P}(s, \ell)|, |\mathcal{P}'(s, \ell)|, |\mathcal{Q}(s, l)|\}$. Moreover, equality is attained only if $\mathcal{F} = \mathcal{Q}(s, \ell)$.
\end{lemma}

\begin{proof}
	Since $\nu(\mathcal{F} \cap {[2\ell+d,n] \choose 3}) \geq 0$, the condition $\nu(\mathcal{F} \cap {[2\ell+d,n] \choose 3}) < c - d + 1$ implies $d \leq c$.
	
	Let $\mathcal{F}' = \mathcal{F} \cap {[2\ell+d, n] \choose 3}$. Since $\nu(\mathcal{F}') \leq c - d$, by Lemma \ref{l.frankl_shifting} we have $|\mathcal{F}'| \leq (c - d){3c - d \choose 2}$. Therefore,
    \begin{align}
        \notag y_{\mathcal{F}}(3) &\geq y_{\mathcal{F}'}(3) \geq {3c-d+1 \choose 3} - (c - d){3c - d \choose 2}\\
        \label{eq_2} &= \frac{(2d+1)(3c-d)(3c-d-1)}{2}.
    \end{align}
    We relax the inequality \eqref{eq: universal d_to_y2} to
    {\small \begin{align}
    	y_{\mathcal{F}}(2) \geq \max\left\{\frac{(4\ell+3c+d-2)(3c-d+1)}{2}, \frac{(\ell+3c-\frac{d-1}{2})(\ell+3c-\frac{d+1}{2})}{2}\right\} \label{eq: relaxed universal d_to_y}
    \end{align}}
    to make the right-hand side of the inequality continuous as a function of $d$.
    Then we consider two cases: whether the minimum is attained on the first or on the second argument.
	
	If the minimum in \eqref{eq: relaxed universal d_to_y} is attained on the first argument, combining \eqref{eq_2} and \eqref{eq: relaxed universal d_to_y}, we have $y_{\mathcal{F}}(2) + y_{\mathcal{F}}(3) \geq f_{\ell,c}(d)$, where $$f_{\ell,c}(d) := \frac{(2d+1)(3c-d)(3c-d-1)}{2} + \frac{(4\ell+3c+d-2)(3c-d+1)}{2}.$$
	Since $f_{\ell,c}''(d) = 6d - 12c + 2 < 0$ for $0 \leq d \leq c$ and $c \geq 1$, we get $y_{\mathcal{F}}(2) + y_{\mathcal{F}}(3) \geq \min(f_{\ell,c}(0), f_{\ell,c}(c))$. Since $f_{\ell,c}(0) = y_{\mathcal{P}'(s, \ell)}(2) + y_{\mathcal{P}'(s, \ell)}(3) + \frac{3c(3c-1)}{2} > y_{\mathcal{P}'(s, \ell)}(2) + y_{\mathcal{P}'(s, \ell)}(3)$ and $f_{\ell,c}(c) = y_{\mathcal{Q}(s, \ell)}(2) + y_{\mathcal{Q}(s, \ell)}(3)$, we get $y_{\mathcal{F}}(2) + y_{\mathcal{F}}(3) \geq \min( y_{\mathcal{P}'(s, \ell)}(2) + y_{\mathcal{P}'(s, \ell)}(3), y_{\mathcal{Q}(s, \ell)}(2) + y_{\mathcal{Q}(s, \ell)}(3))$ and equality holds only if $d=c$. Moreover, $\mathcal{F}$ contains no sets with size less then $2$ and both $\mathcal{P}'(s, \ell)$ and $\mathcal{Q}(s, \ell)$ contain all sets with size at least $4$. Therefore, $|\mathcal{F}| \leq \max(|\mathcal{P}'(s, \ell)|, |\mathcal{Q}(s, \ell)|)$. If equality holds, then $d = c$, consequently, $\nu(\mathcal{F} \cap {[2\ell+c,n] \choose 3}) \leq 0$, that is, $\mathcal{F} \cap {[2\ell+d,n] \choose 3} = \emptyset$ and $\mathcal{F}^{(3)} \subset \mathcal{Q}(s, \ell)^{(3)}$. Moreover, equality also holds in the bound on $y_{\mathcal{F}}(2)$, therefore, $\mathcal{F}^{(2)} = {[2\ell+c-1] \choose 2}$. Since $\mathcal{F}$ contains no singletons, attaining the equality implies that $\mathcal{F} \subset \mathcal{Q}(s, \ell)$ and therefore $\mathcal{F} = \mathcal{Q}(s, \ell)$.
	
	If the minimum in \eqref{eq: relaxed universal d_to_y} is attained on the second argument, we get $y_{\mathcal{F}}(2) + y_{\mathcal{F}}(3) \geq g_{\ell,c}(d)$, where
	$$g_{\ell,c}(d) := \frac{(2d+1)(3c-d)(3c-d-1)}{2} + \frac{(\ell+3c-\frac{d}{2} + \frac{1}{2})(\ell+3c-\frac{d}{2} - \frac{1}{2})}{2}.$$
	Again, we note, that $g''_{\ell,c}(d) = 6d - 12c + \frac{13}{4} < 0$ for $0 \leq d \leq c, c \geq 1$ and therefore $y_{\mathcal{F}}(2) + y_{\mathcal{F}}(3) \geq \min(g_{\ell,c}(0), g_{\ell,c}(c))$. For further analysis, we need to mention that the minimum in \eqref{eq: relaxed universal d_to_y} is attained on the second argument if and only if $\ell \leq 6c+\frac{7}{2}-\frac{5}{8}d$. (To solve the inequality $\frac{(4\ell+3c+d-2)(3c-d+1)}{2} \geq \frac{(\ell+3c-\frac{d}{2} + \frac{1}{2})(\ell+3c-\frac{d}{2} - \frac{1}{2})}{2}$, it is convenient to rewrite it  as $\frac{(n+2k-2)(n-2k+1)}{2} \geq \frac{(n-k+\frac{1}{2})(n-k-\frac{1}{2})}{2}$, where $n=2\ell+3c$ and $k=\ell+\frac{d}{2}$.) 
     Thus, we need to verify that for $c \geq d \geq 5$, the inequalities $y_{\mathcal{F}}(2) + y_{\mathcal{F}}(3) \geq \min\{g_{\ell,c}(0), g_{\ell,c}(c)\}$ and $\ell \leq 6c+\frac{7}{2}-\frac{5}{8}d$ imply $y_{\mathcal{F}}(2) + y_{\mathcal{F}}(3) > y_{\mathcal{P}(s, \ell)}(2) + y_{\mathcal{P}(s, \ell)}(3)$. Indeed,
	\[
	\begin{split}
		&g_{\ell,c}(0) = \frac{3c(3c-1)}{2} + \frac{(\ell+3c+\frac{1}{2})(\ell+3c-\frac{1}{2})}{2} = \\
        &= (y_{\mathcal{P}(s, \ell)}(2) + y_{\mathcal{P}(s, \ell)}(3)) + \frac{3c(3c-1)}{2} - \frac{2\ell+6c+\frac{1}{2}}{2} \geq  \\& \geq (y_{\mathcal{P}(s, \ell)}(2) + y_{\mathcal{P}(s, \ell)}(3)) + \frac{9c^2-21c-\frac{15}{2}}{2} > (y_{\mathcal{P}(s, \ell)}(2) + y_{\mathcal{P}(s, \ell)}(3))
	\end{split}
	\]
	and
	{\small
    \[
	\begin{split}
		g_{\ell,c}(c) &= \frac{(2c+1)(2c)(2c-1)}{6} + \frac{(\ell+\frac{5}{2}c+\frac{1}{2})(\ell+\frac{5}{2}c-\frac{1}{2})}{2}  \\& = (y_{\mathcal{P}(s, \ell)}(2) + y_{\mathcal{P}(s, \ell)}(3)) + \frac{(2c+1)(2c)(2c-1)}{6} - \frac{(2\ell+\frac{11}{2}c + \frac{1}{2})(\frac{c}{2}+\frac{1}{2})}{2}  \\& \geq (y_{\mathcal{P}(s, \ell)}(2) + y_{\mathcal{P}(s, \ell)}(3)) + \frac{(2c+1)(2c)(2c-1)}{6} - \frac{(\frac{35}{2}c + \frac{35}{8})(\frac{c}{2}+\frac{1}{2})}{2} \\ &> (y_{\mathcal{P}(s, \ell)}(2) + y_{\mathcal{P}(s, \ell)}(3)).
	\end{split}
	\]
    }
\end{proof}

\subsection{Large $d$}

In this case, we use Lemmas \ref{l.d_geq_c_to_y3} and \ref{l.d_geq_1_to_y3} to bound $y_{\mathcal{F}}(3)$. As above, the choice between these two bounds depends on the relation between $\ell$ and $c$. In order to bound $y_{\mathcal{F}}(2)$, we use Corollary \ref{c.y2}. Together, Lemmas \ref{l.d_geq_1_to_y3} and \ref{l.d_geq_c_to_y3} imply that if $c \geq 3$ and $c + 2 \leq d(\mathcal{F}) \leq 2c-1$, then $|\mathcal{F}| < \max(|\mathcal{P}(s,\ell)|, |\mathcal{P}'(s, \ell)|)$. Recall that by \eqref{eqd2c}, $d\le 2c$. We analyze separately the last remaining case $d = 2c$. This case is easy, since by Lemma \ref{l.big_X} the condition $d=2c$ implies existing of many $(s-1)$-matchings and $(c+2)$-element complement of the union of each such matching can not contain any sets from $\mathcal{F}$.

Note that in the case $d \geq c + 2$ we only require the condition $c \geq 3$ instead of $c \geq 5$. Therefore, dealing with the case $c=3$ in the Appendix, we may assume $d \leq 4$.

\begin{lemma} \label{l.no_big_d}
	Let $c \geq 3$ and $d \geq c + 2$. If $\mathcal{F}$ is a shifted family without one-element sets such that $d(\mathcal{F}) = d$ and $\nu(\mathcal{F}) < s$, then $|\mathcal{F}| < \max\{|\mathcal{P}(s, \ell)|, |\mathcal{P'}(s, \ell)|\}$.
\end{lemma}

\begin{proof}
By Corollary \ref{c.y2} either
\begin{equation}\label{eqlem291}
	    |\overline{\mathcal{P'}(s, \ell)}| - (y_{\mathcal{F}}(0) + y_{\mathcal{F}}(1) + y_{\mathcal{F}}(2)) \leq \frac{(4\ell+d-3)d}{2}
\end{equation} or
{\small
\begin{equation}\label{eqlem292}
	      |\overline{\mathcal{P}(s, \ell)}| - (y_{\mathcal{F}}(0) + y_{\mathcal{F}}(1) + y_{\mathcal{F}}(2)) \leq \frac{(2\ell+6c-\lceil\frac{d}{2}\rceil + 1)\lceil\frac{d}{2}\rceil}{2} \leq  \frac{(2\ell+5c + 1)c}{2}. 
\end{equation}} 
In the last inequality, we used that that the function $f(x) = (2\ell+6c-x + 1)x$ is increasing for $x < \ell + 3c + \frac{1}{2}$, and that the argument $d$ is bounded: $d \leq 2c$ by \eqref{eqd2c}. 
	
	Assume first that \eqref{eqlem292} holds. 
    If $\ell < \frac{2c^2-5c+1}{c}$
    then we use Lemma \ref{l.d_geq_1_to_y3} and get
	\begin{align*}y_{\mathcal{F}}(3) &\geq {3c-1 \choose 2} = \frac{(2\frac{2c^2-5c+1}{c}+5c+1)c}{2} > \frac{(2\ell+5c+1)c}{2}\\ &\geq |\overline{\mathcal{P}(s, \ell)}| - (y_{\mathcal{F}}(0) + y_{\mathcal{F}}(1) + y_{\mathcal{F}}(2))\end{align*}
	and therefore $|\overline{\mathcal{F}}| \geq y_{\mathcal{F}}(0) + y_{\mathcal{F}}(1) + y_{\mathcal{F}}(2) + y_{\mathcal{F}}(3) > |\overline{\mathcal{P}(s, \ell)}|$.
	
	If $\ell > \frac{c^2+3c}{6c-4}$ then we use Lemma \ref{l.d_geq_c_to_y3} and get
	$$y_{\mathcal{F}}(3) \geq (2\ell+c)(2c-1) = \frac{(2\ell+5c+1)c}{2} + (3c-2)\ell - \frac{c^2+3c}{2} > \frac{(2\ell+5c+1)c}{2}$$
	and therefore $|\overline{\mathcal{F}}| > |\overline{\mathcal{P}(s, \ell)}|$.
	
	Since for $c \geq 3$ we have $\frac{c^2+3c}{6c-4} < \frac{2c^2-5c+1}{c}$, at least one of the inequalities $\ell < \frac{2c^2-5c+1}{c}$ and $\ell > \frac{c^2+3c}{6c-4}$ is satisfied. Therefore, if \eqref{eqlem292} holds then $|\mathcal{F}| < |\mathcal{P}(s, \ell)|$.
	
	If \eqref{eqlem291} holds and $d \leq 2c - 1$, then we use Lemma \ref{l.d_geq_c_to_y3} to bound $y_{\mathcal{F}}(3)$ and get
	{\small $$y_{\mathcal{F}}(3) \geq (2\ell+c)(2c-1) > \frac{(4\ell+2c-4)(2c-1)}{2} \geq |\overline{\mathcal{P'}(s, \ell)}| - (y_{\mathcal{F}}(0) + y_{\mathcal{F}}(1) + y_{\mathcal{F}}(2))$$}
	and therefore $|\mathcal{F}| < |\mathcal{P'}(s, \ell)|$.
	
	The only remaining case is when $d = 2c$ and \eqref{eqlem291} holds. The RHS of \eqref{eqlem291} is $(4\ell+2c-3)c$ in this case. By Lemma \ref{l.big_X}, for any $x$ in $[2, 2s - 1]$ the family $\mathcal{F} \cup {[2s-1] \setminus \{x\} \choose 2}$ contains an $(s-1)$-matching. Therefore, for any $x$ in $[2, 2s - 1]$ and $\{y_1, y_2\} \in {[2s, n] \choose 2}$ we get $\{x, y_1, y_2\} \notin \mathcal{F}$. So $y(3) \geq (2s-2){c + 1 \choose 2} = c(c+1)\ell + c^3 - c$. For $c \geq 3$ we have $c(c+1)l + c^3 - c > = 4c\ell+(2c-3)c$ and therefore $|\mathcal{F}| < |\mathcal{P'}(s, \ell)|$.
\end{proof}


\subsection{Proof of Theorem \ref{t.relaxed} for $c \geq 5$}
Let $\mathcal{F}$ be a shifted up-set without one-element sets with $\nu(\mathcal{F}) < s$.

If $d(\mathcal{F}) = 0$, then by Lemma \ref{l.even_d_to_y2} we have $y_{\mathcal{F}}(2) \geq \min\{y_{\mathcal{P}(s, \ell)}(2), y_{\mathcal{P}'(s,\ell)}(2)\}$ and equality is achieved only if $\mathcal{F}^{(2)} = \mathcal{P}(s, \ell)^{(2)}$ or $\mathcal{F}^{(2)} = \mathcal{P}'(s, \ell)^{(2)}$. Since $\mathcal{F}$ does not contain any one-element sets and both $\mathcal{P}(s, \ell)$ and $\mathcal{P}'(s, \ell)$ contain all sets with size at least $3$, it implies that $|\mathcal{F}| \leq \max\{|\mathcal{P}(s, \ell)|, |\mathcal{P}'(s, \ell)|\}$ and equality is attained only if $\mathcal{F} = \mathcal{P}(s, \ell)$ or $\mathcal{F} = \mathcal{P}'(s, \ell)$.

If $1 \leq d(\mathcal{F}) \leq 4$, then by Lemmas \ref{l.d_1_2} and \ref{l.d_3_4} we have $|\mathcal{F}| < \max\{|\mathcal{P}(s,\ell)|,$ $ |\mathcal{P'}(s, \ell)|, |\mathcal{Q}(s, \ell)|\}$.

If $5 \leq d(\mathcal{F}) \leq c + 1$ and $\nu(\mathcal{F}) \cap {[2\ell+d, n] \choose 3} \geq c - d + 1$, then Lemmas~\ref{l.moderate_odd_d} and~\ref{l.moderate_even_d} imply $|\mathcal{F}| < \max\{|\mathcal{P}(s,\ell)|, |\mathcal{P'}(s, \ell)|\}$.

If $5 \leq d(\mathcal{F}) \leq c + 1$ and $\nu(\mathcal{F}) \cap {[2\ell+d, n] \choose 3}) < c - d + 1$, then by Lemma~\ref{l.big_nu} we get $|\mathcal{F}| \leq \max\{|\mathcal{P}(s,\ell)|, |\mathcal{P'}(s, \ell)|, |\mathcal{Q}(s, \ell)|\},$ and equality is achieved only if $\mathcal{F} = \mathcal{Q}(s, \ell)$.

If $d(\mathcal{F}) \geq c + 2$, then Lemma \ref{l.no_big_d} implies that $|\mathcal{F}| < \max\{|\mathcal{P}(s, \ell)|, |\mathcal{P'}(s, \ell)|\}$.


\section{Appendix A: proofs of some technical statements}
We deferred to the appendix some of the more technical/computational proofs. We restate the claims for convenience.



\claimeleven*

\begin{proof}
    Put
    \begin{align} \label{eq:Ndef}
      N(s,\ell):=\min\big\{ |\overline{\mathcal{P'}(s, \ell)}|, |\overline{\mathcal{P}(s, \ell)}|, |\overline{\mathcal{Q}(s, l)}|, |\overline{\mathcal{W}(s, \ell)}|\big\}  
    \end{align}
    and
    \begin{align} \label{eq:Kdef}
        K(s,\ell):=\min\Big\{ \big|\overline{\mathcal{P'}(s, \ell)}^{(\leq 3)}\big|, \big|\overline{\mathcal{P}(s, \ell)}^{(\leq 3)}\big|, \big|\overline{\mathcal{Q}(s, \ell)}^{(\leq 3)}\big|, \big|\overline{\mathcal{W}(s, \ell)}^{(\leq 3)}\big|\Big\}.
    \end{align}
    We need to prove
    \begin{align} \label{eq: N_ineq}
        N(s-1,\ell-2)>N(s,\ell)
    \end{align}
    and
    \begin{align} \label{eq: K_ineq}
        K(s-1,\ell-2)>K(s,\ell).
    \end{align}

    Since $\mathcal{P}, \mathcal{P}'$, and $\mathcal{Q}$ contain all sets with size at least $4$,
    $$K(s,\ell)=\min\Big\{ |\overline{\mathcal{P'}(s, \ell)}|, |\overline{\mathcal{P}(s, \ell)}|, |\overline{\mathcal{Q}(s, \ell)}|, \big|\overline{\mathcal{W}(s, \ell)}^{(\leq 3)}\big|\Big\}.$$

    Straightforward calculations give
    \begin{align} \label{eq:W_3_layers}
        \big|\overline{\mathcal{W}(s, \ell)}^{(\leq 3)}\big| = (c^2+3c+4)s + \frac{c^3 - c}{6}.
    \end{align}

    Let us show that we can exclude  $\big|\overline{\mathcal{W}(s, \ell)}^{(\leq 3)}\big|$ from the RHS of \eqref{eq:Kdef} without changing the minimum. Indeed, combining \eqref{eq:W_3_layers} with \eqref{eq_P'_complement} and \eqref{eq_Q_complement}, we get
    \begin{align} \label{eq: W minus P}
        \big|\overline{\mathcal{W}(s, \ell)}^{(\leq 3)}\big| - |\overline{\mathcal{P}(s, \ell)}| = (c^2 - 3c)s + \frac{c^3+9c^2+14c}{6}
    \end{align}
    and
    \begin{align} \label{eq: W minus Q}
        \big|\overline{\mathcal{W}(s, \ell)}^{(\leq 3)}\big| - |\overline{\mathcal{Q}(s, \ell)}| = c(c-1)(6s-7c-7).
    \end{align}

    For $c \geq 3$, as well as for $c=2, s \leq 6$, equality \eqref{eq: W minus P} implies that $\big|\overline{\mathcal{W}(s, \ell)}^{(\leq 3)}\big| \geq |\overline{\mathcal{P}(s, \ell)}|$. For $c = 1$, as well as for $c = 2, s \geq 4$,  equality \eqref{eq: W minus Q} implies that $\big|\overline{\mathcal{W}(s, \ell)}^{(\leq 3)}\big| \geq |\overline{\mathcal{Q}(s, \ell)}|$. These two cases combined imply that for all $s$ and $c$ we have $\big|\overline{\mathcal{W}(s, \ell)}^{(\leq 3)}\big| \geq \min\left\{|\overline{\mathcal{P}(s, \ell)}|, |\overline{\mathcal{Q}(s, \ell)}|\right\}$. 
    Consequently, we have $|\overline{\mathcal{W}(s, \ell)}|\geq \min\left\{|\overline{\mathcal{P}(s, \ell)}|, |\overline{\mathcal{Q}(s, \ell)}|\right\}$, and thus we can as well exclude $|\overline{\mathcal{W}(s, \ell)}|$ from the RHS of  \eqref{eq:Ndef} without changing the respective minimum. 
    
    Thus, to prove both \eqref{eq: K_ineq} and \eqref{eq: N_ineq}, it is sufficient to prove
    $$M(s-1,\ell-2)>M(s,\ell),$$
    where
    \begin{align} \label{eq:Mdef}
      M(s,\ell):=\min\big\{ |\overline{\mathcal{P'}(s, \ell)}|, |\overline{\mathcal{P}(s, \ell)}|, |\overline{\mathcal{Q}(s, \ell)}|\big\}.
    \end{align}

    Using straightforward calculations, one can check that for $\ell \geq 2$ the inequality $|\overline{\mathcal{P'}(s, \ell)}| \leq |\overline{\mathcal{P}(s, \ell)}|$ is equivalent to $s \geq 7c + 2$. 
    Indeed, we have 
\begin{align*}
&|\overline{\mathcal{P'}(s', \ell')}| - |\overline{\mathcal{P}(s', \ell')}| = {n \choose 2} - {2\ell-1 \choose 2} - {n-\ell+1 \choose 2} = \\
&= \frac{1}{2}\left(n^2 - n - 2(\ell-1)(2\ell-1) - n^2 + 2\ell n - n +\ell(\ell-1) \right) = \\
&= \frac{1}{2}\left( 2(\ell - 1)n - (5\ell-2)(\ell-1) \right) = (\ell-1)\left(n - \frac{5\ell-2}{2}\right) = \\
&= (\ell-1)\left(2s+c - \frac{5s-5c-2}{2}\right) = \frac{\ell-1}{2} \left(7c+2-s\right).
\end{align*}

Moreover, for $l = 2$ we have $|\overline{\mathcal{P'}(s', \ell')}| = |\overline{\mathcal{P}(s', \ell')}| = {n \choose 2} + n + 1$.

    
Therefore, if the minimum in the definition of $M(s-1, l-2)$ is achieved only on the first argument, then the inequality $$|\overline{\mathcal{P'}(s-1, \ell-2)}| < |\overline{\mathcal{P}(s-1, \ell-2)}|$$ implies $s > 7(c+1)+3$ and we have
\[
\begin{split}
	M(s-1,\ell-2) &= |\overline{\mathcal{P'}(s-1, \ell-2)}| \\& = (6c+10)(s-1)-\frac{3}{2}(c+1)^2-\frac{5}{2}(c+1)\\
    &= (6c+4)s-\frac{3}{2}c^2-\frac{5}{2}c + 6s - 9c - 14 \\ &>(6c+4)s-\frac{3}{2}c^2-\frac{5}{2}c = |\overline{\mathcal{P'}(s, l)}| \geq M(s,\ell).
\end{split}
\]

If the minimum is achieved on the second argument, we have
\[
\begin{split}
	M(s-1,\ell-2) &= |\overline{\mathcal{P}(s-1, \ell-2)}| \\
	&={s+2c+2 \choose 2} + n = {s+2c+1 \choose 2} + n + 1 + s + 2c \\
    &> {s+2c+1 \choose 2} + n + 1 =  |\overline{\mathcal{P}(s, \ell)}| \geq M(s,\ell).
\end{split}
\]

If the minimum is achieved on the third argument, we have
\[
\begin{split}
	M(s-1,\ell-2) &= |\overline{\mathcal{Q}(s-1, \ell-2)}| \\
	&=(4c+8)(s-1) + \frac{4(c+1)^3-4(c+1)}{3} \\
    &= (4s+4c^2-8) + \left((4c+4)s + \frac{4c^3-4c}{3}\right) \\ &> (4c+4)s + \frac{4c^3-4c}{3} = |\overline{\mathcal{Q}(s, \ell)}| \geq M(s,\ell).
\end{split}
\]


\end{proof}

\claimfifteen*

\begin{proof}	
	Since $(4\ell+3c+d-2)(3c-d+1) = (4\ell+3c+d-7)(3c-d+1) + 5(3c-d+1)$ and $(4\ell+3c+d-7)(3c-d+2) = (4\ell+3c+d-7)(3c-d+1) + (4\ell+3c+d-7)$, inequality \eqref{eq: odd_d_to_y2: 1} is equivalent to $5(3c-d+1) < 4\ell+3c+d-7$. Simplifying the inequality, we get
	$$\ell > 3c-\frac{3}{2}d+3.$$
	
	To solve \eqref{eq: odd_d_to_y2: 2}, it is convenient to note that it turns into an equality for $\ell+\frac{d+1}{2} = 3$, since both the left-hand side and the right-hand side are equal to the number of two-element sets that can be shifted to $\{3, 4\}$. That is, the difference between the left-hand side and the right-hand side equals zero when $\ell+\frac{d}{2}-\frac{5}{2} = 0$. Therefore, $(\ell+\frac{d}{2}-\frac{5}{2})$ divides the difference between the sides of \eqref{eq: odd_d_to_y2: 2}. Noting it, it is easy to factor $$\frac{(\ell+3c-\frac{d-1}{2})(\ell+3c-\frac{d+1}{2})}{2} - \frac{(4\ell+3c+d-7)(3c-d+2)}{2}$$ and get that \eqref{eq: odd_d_to_y2: 2} is equivalent to
	$$\left(\ell+\frac{d}{2}-\frac{5}{2}\right)\left(\ell+\frac{5}{2}d-6c-\frac{11}{2}\right) < 0.$$
	Since $\left(\ell+\frac{d}{2}-\frac{5}{2}\right) > 0$, the inequality is equivalent to $$\ell < 6c-\frac{5}{2}d+\frac{11}{2}.$$
	
	Therefore, if both inequalities \eqref{eq: odd_d_to_y2: 1} and \eqref{eq: odd_d_to_y2: 2} are not satisfied, then $6c-\frac{5}{2}d+\frac{11}{2} \leq \ell \leq 3c-\frac{3}{2}d+3$. It is equivalent to $d \geq 3c+\frac{5}{2}$ and contradicts the assumption $d \leq 2c$.
\end{proof}

\corsixteen*



\begin{proof}
	Since $\mathcal{F}$ is an up-set without one-element sets, $y_{\mathcal{F}}(0) = 1$ and $y_{\mathcal{F}}(1) = n$. By Lemmas~\ref{l.even_d_to_y2} and~\ref{l.odd_d_to_y2} we get $$y(2) \geq \max\big\{\frac{(4\ell+3c+d-2)(3c-d+1)}{2}, \frac{(\ell+3c-\lceil \frac{d}{2} \rceil + 1)(\ell+3c-\lceil \frac{d}{2} \rceil)}{2}\big\}.$$ If $y(2) \geq \frac{(4\ell+3c+d-2)(3c-d+1)}{2},$ then
	
	\[
	\begin{split}
		&|\overline{\mathcal{P'}(s, \ell)}| - (y_{\mathcal{F}}(0) + y_{\mathcal{F}}(1) + y_{\mathcal{F}}(2)) \\& \leq (6c+4)(\ell+c) - \frac{3}{2}c^2-\frac{5}{2}c - \left(1 + 2\ell + 3c + \frac{(4\ell+3c+d-2)(3c-d+1)}{2}\right) \\& = \frac{(4\ell+d-3)d}{2}.
	\end{split}
	\]
	
	If $y(2) \geq \frac{(\ell+3c-\lfloor \frac{d}{2} \rfloor + 1)(\ell+3c-\lfloor \frac{d}{2} \rfloor)}{2}$, then
	
	\[
	\begin{split}
		&|\overline{\mathcal{P}(s, \ell)}| - (y_{\mathcal{F}}(0) + y_{\mathcal{F}}(1) + y_{\mathcal{F}}(2)) \leq \\& \leq (1 + n + {\ell+3c+1 \choose 2}) - (1 + n + {\ell+3c-\lceil \frac{d}{2} \rceil + 1 \choose 2}) = \\& = \frac{(2\ell+6c-\lceil\frac{d}{2}\rceil + 1)\lceil\frac{d}{2}\rceil}{2}.
	\end{split}
\]
\end{proof}

\lemmatwseven*

\begin{proof}
    As in Lemma \ref{l.moderate_odd_d}, we establish \eqref{eq_desired_y3_P'} and \eqref{eq_desired_y3_P}.
	
	To prove \eqref{eq_desired_y3_P'}, we consider cases $\ell \leq 2c-4$ and $\ell \geq 2c-3$. If $\ell \leq 2c-4$, we use Lemma \ref{l.d_geq_1_to_y3} and get
	$$y(3) \geq {3c-1 \choose 2} = \frac{(4(2c-4+\frac{8}{c+1})+(c+1)-3)(c+1)}{2} > \frac{(4\ell+d-3)d}{2}.$$
	If $\ell \geq 2c-3$, then $\ell \geq 2d-5$ (since $d \leq c + 1$), so we get $$\frac{(d-2)^2}{2\ell(d-1)} \leq \frac{(d-2)^2}{2(2d-5)(d-1)} = \frac{d^2-4d+4}{4d^2-14d+10} \leq \frac{1}{4}.$$
	(In the last inequality we used $d \geq 4$.) By Corollary \ref{c.even_d_to_y3} we get
	\begin{align*}y(3) &\geq (2\ell+d-2)(2d-3)\frac{1}{1+\frac{(d-2)^2}{2\ell(d-1)}} \\
    &\geq (2\ell+d-2)(2d-3)\frac{1}{1+\frac{1}{4}}\\
    &\geq (2\ell+d-2)d > \frac{(4\ell+d-3)d}{2}.\end{align*}
	
	In order to prove \eqref{eq_desired_y3_P} we consider cases $\ell \leq 6c-18$ and $\ell \geq 6c-17$.
	If $\ell \leq 6c-18$, we use Lemma \ref{l.d_geq_1_to_y3} and get
	$$y(3) \geq {3c-1 \choose 2} = \frac{(2(6c-18+\frac{20}{c+1})+6c)\frac{c+1}{2}}{2} > \frac{(2\ell+6c-\frac{d}{2}+1)(\frac{d}{2})}{2}.$$
	If $\ell \geq 6c-17$, then $\ell\geq 6d - 23 \geq 2d$ and therefore $\frac{(d-2)^2}{2\ell(d-1)} \leq \frac{(d-2)^2}{2d(d-1)} < \frac{1}{4}$. Straightforward computations show that for all $\ell \geq 0$ and $d \geq 6$ an inequality $\frac{4}{5}(2\ell+d-2)(2d-3) \geq \frac{(3\ell+18)d}{4}$ holds. Combining these inequalities with Corollary~\ref{c.even_d_to_y3}, we get
    \begin{align*} 
    y(3) &\geq (2\ell+d-2)(2d-3)\frac{1}{1+\frac{(d-2)^2}{2\ell(d-1)}} > \frac{4}{5}(2\ell+d-2)(2d-3)  \\
    &\geq \frac{(3\ell+18)d}{4}\geq \frac{(2\ell+6c-\frac{d}{2}+1)\frac{d}{2}}{2}.\end{align*}
\end{proof}

\section{Appendix B: Proof of Theorem~\ref{t.relaxed} for $c\le 4$}

\subsection{One more bound on $y(3)$}

For small $c$ we will usually use the following lemma instead of Lemmas~\ref{l.odd_d_to_y3} and~\ref{l.even_d_to_y3} to get a bound on $y(3)$.


\begin{lemma} \label{l.old_d_to_y3}
	Let $d$ be an integer, $2 \leq d \leq c$. Let $\mathcal{F} \subset 2^{[n]}$ be a shifted family with $\nu(\mathcal{F}) < s$. Suppose that for all $i \in [\ell]$ we have $\{i, 2\ell+d+1-i\} \in \mathcal{F}$ and  $\nu(\mathcal{F} \cap {[2\ell+d+1, n] \choose 3}) \geq c - d$. Then
	$$y_{\mathcal{F}}(3) \geq (2d-1)(s+2c-2d) + 1.$$
	Moreover, equality is achieved only if $d = 2$ and $\mathcal{F}^{(3)}$ contains all $3$-element sets that cannot be shifted to $\{\ell + 2, n - 2, n - 1\}$.
\end{lemma}

Note that the condition $\nu(\mathcal{F} \cap {[2\ell+d+1, n]}) \geq c - d$ is the same as in Lemmas \ref{l.odd_d_to_y3} and \ref{l.even_d_to_y3} for $d(\mathcal{F}) = d + 1$ and the condition $\{i, 2\ell+d+1-i\} \in \mathcal{F}$ for all $i \in [\ell]$ is similar to $d(\mathcal{F}) > d$ for even $d$, but is weaker, since we do not require $\{i, 2\ell+d+1-i\} \in \mathcal{F}$ for $i \geq \ell + 1$. On the other hand, we get a bound on $y(3)$ that is approximately twice worse. Therefore, we use Lemma \ref{l.old_d_to_y3} when we can provide the condition $\{i, 2\ell+d+1-i\} \in \mathcal{F}$ for all $i \in [\ell]$, but we do not have this condition for $i > \ell$. It is a rather common situation for small $c$.


\begin{proof}
	Since $\mathcal{F}$ is shifted and $\nu(\mathcal{F} \cap {[2\ell+d+1, n] \choose 3}) \geq c - d$, we have $\nu(\mathcal{F} \cap {[2\ell+d+1, 2\ell+3c-2d] \choose 3}) \geq c - d$. Note, that $2\ell+3c-2d = n - 2d$. Consider $d$ sets $\{\ell + i, n - 2d + i, n - i + 1\}, i \in [d]$. Together with the sets $\{i, 2\ell + d + 1 - i\}$ for $i \in [\ell]$ and the  $(c-d)$-matching in $\mathcal{F} \cap {[2\ell+d+1, 2\ell+3c-2d] \choose 3}$, these sets form an $s$-matching. Since $\nu(\mathcal{F}) < s$ and all sets from the other two groups are in $\mathcal{F}$, we have $\{\ell + i, n - 2d + i, n - i + 1\} \notin \mathcal{F}$ for some $i \in [d]$. Since $\mathcal{F}$ is shifted, all sets, that may be shifted to $\{l + i, n - 2d + i, n - i + 1\}$ are also not in $\mathcal{F}$.
	
	Next, we estimate the minimum over $i$ of the number of sets that can be shifted to $\{\ell + i, n - 2d + i, n - i + 1\}$. For all $x \in [\ell + i, n - 2d + i - 1]$ and all $\{y, z\} \in {[n - 2d + i, n] \choose 2}$ such that $\{y, z\}$ may be shifted to $\{n - 2d + i, n - i + 1\}$ the set $\{x, y, z\}$ may be shifted to $\{\ell + i, n - 2d + i, n - i + 1\}$. We have $(n - 2d + i - 1) - (\ell + i - 1) = s + 2c - 2d$ such $x$ and ${2d - i + 1 \choose 2} - {2d - 2i + 1 \choose 2} = \frac{(4d-3i+1)i}{2}$ such pairs $\{y, z\}$, that is, $(s + 2c - 2d)\frac{(4d-3i+1)i}{2}$ such triples $\{x, y, z\}$. All these triples are different, since we count unordered pairs $x, y$ and for all triples $x \leq n - 2d + i - 1$ and $\min{(y, z)} \geq n - 2d + i$. Moreover, we did not count a set $\{n - 2, n - 1, n\}$, that of course may be shifted to $\{\ell + i, n - 2d + i, n - i + 1\}$, because in all considered triples $x \leq n - 2d + i - 1 \leq n - d - 1 \leq n - 3$ (here we use $d \geq 2$). If $d > 2$ or $i < d$, we also did not count the set $\{n - 3, n - 1, n\}$, that also may be shifted to $\{\ell + i, n - 2d + i, n - i + 1\}$. Therefore, we have at least $$(s + 2c - 2d)\frac{(4d-3i+1)i}{2} + 1$$ sets that may be shifted to $\{\ell + i, n - 2d + i, n - i + 1\}$ and this bound is strict only if $d=2$ and $i=2$.
	
	Finally, we apply a usual convexity argument. Since for some $i \in [d]$ we have $\{\ell + i, n - 2d + i, n - i + 1\} \notin \mathcal{F}$ and at least $$(s + 2c - 2d)\frac{(4d-3i+1)i}{2} + 1$$ sets may be shifted to $\{l + i, n - 2d + i, n - i + 1\} \notin \mathcal{F}$, we have $$y_{\mathcal{F}}(3) \geq \min_{i\in[d]}\Big\{(s + 2c - 2d)\frac{(4d-3i+1)i}{2} + 1\Big\}.$$ This bound is an upward convex function, the minimum is attained at one of the ends. For $d \geq 2$ we have $\frac{(d+1)d}{2} \geq 2d - 1$, and so the minimum is attained for $i=1$ and we have the desired bound $y_{\mathcal{F}}(3) \geq (2d-1)(s+2c-2d) + 1$. Equality is attained only if $d=i=2$. Moreover, if $\ff$ attains equality then it must contain all $3$-element sets that cannot be shifted to 
    $\{\ell + i, n - 2d + i, n - i + 1\}$.
    
\end{proof}

We go to the proof of the theorem. We start with the case $c = 2$, since it is the best to illustrate methods that we use for small $c$. 

\subsection{$c=2$}

For $c=2$ we have
$$y_{\mathcal{P}(s, \ell)}(2) = {s + 5 \choose 2} = \frac{s^2+9s+20}{2},$$
$$y_{\mathcal{P}'(s, \ell)}(2) = {2s + 2 \choose 2} - {2s - 5 \choose 2} = 14s-14$$ and $$y_{\mathcal{Q}(s, \ell)}(2) + y_{\mathcal{Q}(s, \ell)}(3) = 10s+5.$$
Straightforward computations show, that the condition $|\mathcal{P}'(s, \ell)| \geq |\mathcal{P}(s, \ell)|$ is equivalent to $s \geq 16$ and the condition $|\mathcal{P}'(s, \ell)| \geq |\mathcal{Q}(s, \ell)|$ is equivalent to $s \leq \frac{19}{4}$. Therefore, $\max\{|\mathcal{P}(s, \ell)|, |\mathcal{P}'(s, \ell)|, |\mathcal{Q}(s, \ell)|\}$ cannot be attained in the second argument. However, it is still more convenient to prove $|\mathcal{F}| \leq \max\{|\mathcal{P}(s, \ell)|, |\mathcal{P}'(s, \ell)|, |\mathcal{Q}(s, \ell)|\}$ than $|\mathcal{F}| \leq \max\}|\mathcal{P}(s, \ell)|, |\mathcal{Q}(s, \ell)|\}$.

In this subsection, we prove the following lemma.

\begin{lemma} \label{l.c_eq_2}
	Let $n = 2s+2$ and $\mathcal{F} \subset 2^{[n]}$ is a shifted up-set with $\nu(\mathcal{F}) < s$ and $\mathcal{F} \cap {[n] \choose 1} = \emptyset$. Then $|\mathcal{F}| \leq \max\{|\mathcal{P}(s, \ell)|, |\mathcal{P}'(s, \ell)|, |\mathcal{Q}(s, \ell)|\}$. Moreover, equality is achieved only if $\mathcal{F}$ is one of the families $\mathcal{P}(s, \ell), \mathcal{P}'(s, \ell), \mathcal{Q}(s, \ell)$.
\end{lemma}

\begin{proof}
	The case $\ell \leq 1$, that is, $s \leq 3$ is treated in \cite{Kl}, so we will assume $s \geq 4$.
	
	Since $\mathcal{F}$ contains no singletons, the inequality $$|\mathcal{F}| \leq \max\{|\mathcal{P}(s, \ell)|, |\mathcal{P}'(s, \ell)|, |\mathcal{Q}(s, \ell)|\}$$ is a consequence of $$y_{\mathcal{F}}(2) + y_{\mathcal{F}}(3) \geq \min\{y_{\mathcal{P}(s, \ell)}(2), y_{\mathcal{P}'(s, \ell)}(2), y_{\mathcal{Q}(s, \ell)}(2) + y_{\mathcal{Q}(s, \ell)}(3)\}.$$	
	If $d(\mathcal{F}) = 0$, then by Lemma \ref{l.even_d_to_y2} we get that $y_{\mathcal{F}}(2) \geq \max\{y_{\mathcal{P}(s, \ell)}(2), y_{\mathcal{P}'(s, \ell)}(2)\}$ and equality is achieved only if $\mathcal{F}$ is one of the families $\mathcal{P}(s, \ell)$ and $\mathcal{P}'(s, \ell)$. Therefore, we may assume $d(\mathcal{F}) > 0$. By Lemma \ref{l.d_geq_1_to_y3} it implies $y_{\mathcal{F}}(3) \geq 10$.
	
	If $\{1, 2s-2\} \notin \mathcal{F}$, then by Lemma \ref{l.EG} we get $y_{\mathcal{F}} \geq 10s-5$ and therefore
	$$y_{\mathcal{F}}(2) + y_{\mathcal{F}}(3) \geq 10s+5 = y_{\mathcal{Q}(s, \ell)}(2) + y_{\mathcal{Q}(s, \ell)}(3).$$
	Moreover, equality is achived only if $\mathcal{F}^{(2)} = {[2s-3] \choose 2}$ and all sets in ${[n] \choose 3} \setminus {[2s-3, n] \choose 3}$ are in $\mathcal{F}$. This conditions together with $\nu(\mathcal{F}) < s$ imply $\mathcal{F} \subset \mathcal{Q}(s, \ell)$. Indeed, if $\mathcal{F}$ contains a set $F_1 \in {[2s-2, n] \choose 3}$ then $\mathcal{F}$ contains an $s$-matching $\{\{2i-1, 2i\} : i \in [s-3]\} \cup \{\{2s-5, 2s-3\}, F_1, F_2\}$, where $F_2 = [2s-2, 2s+2] \setminus F_1 \cup \{2s-4\}$. Summarizing, if $|\mathcal{F}| \geq \max\{|\mathcal{P}(s, \ell)|, |\mathcal{Q}(s, \ell)|\}, d(\mathcal{F}) > 0$, and $\{1, 2s-2\} \notin \mathcal{F}$, then $\mathcal{F} = \mathcal{Q}(s, \ell)$. 
    
    In what follows, we assume $\{1, 2s-2\} \in \mathcal{F}$. Our goal is to prove that  $|\mathcal{F}| < \max\{|\mathcal{P}(s, \ell)|, |\mathcal{Q}(s, \ell)|\}$, that is, $y_{\mathcal{F}}(2) + y_{\mathcal{F}}(3) > \min\{y_{\mathcal{P}(s, \ell)}(2), y_{\mathcal{Q}(s, \ell)}(2) + y_{\mathcal{Q}(s, \ell)}(3)\}$.
	
	If for some $i \in [2, s-2]$ we have $\{i, 2s-1-i\} \notin \mathcal{F}$, then by Lemma \ref{l.EG} we get $y_{\mathcal{F}}(2) \geq \frac{(4s-3i+1)(i+4)}{2}$ and by convexity $y_{\mathcal{F}}(2) \geq \min\big\{3(4s-5), \frac{(s+7)(s+2)}{2}\big\}$. Combining it with the inequality $y_{\mathcal{F}}(3) \geq 10$, we get $$y_{\mathcal{F}}(2) + y_{\mathcal{F}}(3) \geq \min\Big\{12s - 5, \frac{s^2 + 9s + 34}{2}\Big\}.$$
	
	Since $\frac{s^2 + 9s + 34}{2} > \frac{s^2 + 9s + 20}{2}$ for all $s$, $12s - 5 > 10s + 5$ for $s \geq 6$ and $12s - 5 > \frac{s^2+9s+20}{2}$ for $s \in \{4, 5\}$, we get
	\[
	\begin{split}
		y_{\mathcal{F}}(2) + y_{\mathcal{F}}(3) &\geq \min(12s - 5, \frac{s^2 + 9s + 34}{2}) > \min\Big\{10s + 5, \frac{s^2+9s+20}{2}\Big\}  \\& = \min\Big\{y_{\mathcal{P}(s, \ell)}(2), y_{\mathcal{Q}(s, \ell)}(2) + y_{\mathcal{Q}(s, \ell)}(3)\Big\}.
	\end{split}
	\]
	
	Thus, we get that if for some $i \in [2, s-2]$ we have $\{i, 2s-1-i\} \notin \mathcal{F}$, then $|\mathcal{F}| < \max\{|\mathcal{P}(s, \ell)|, |\mathcal{Q}(s, \ell)|\}$. In what follows, we  assume $\{i, 2s-1-i\} \in \mathcal{F}$ for all $i \in [1, s-2]$. (The case $i=1$ was treated in the previous step of the proof). Note that, unlike in the case of big $c$, we cannot use the same argument to prove $\{s-1, s\} \in \mathcal{F}$, since an inequality $y_{\mathcal{F}}(2) + y_{\mathcal{F}}(3) \geq \frac{(s+4)(s+3)}{2} + 10$ does not imply $y_{\mathcal{F}}(2) + y_{\mathcal{F}}(3) > \min\big\{10s + 5, \frac{s^2+9s+20}{2}\}$ for $s \in \{6, 7, 8\}$. This is the reason why do we have to apply the weaker Lemma~\ref{l.old_d_to_y3} instead of Lemma~\ref{l.odd_d_to_y3} that  we used for large $c$.
	
	By Lemma \ref{l.old_d_to_y3} we get $y_{\mathcal{F}}(3) \geq 3s + 1$. With this upgraded bound on $y_{\mathcal{F}}(3)$, we turn back to the analysis of $\mathcal{F}^{(2)}$.
	
	The next step is to prove that $\{i, 2s+1-i\} \in \mathcal{F}$ for all $i \in [2, s]$. That is, we add a set $\{s-1, s\}$ to the $(s-2)$-matching in $\mathcal{F}^{(2)}$ from the previous paragraph and move the matching to the right by one element.
	
	We use completely the same argument as in the previous step. The only difference is that we get a worse bound on $y(2)$, but use a better bound on $y(3)$. If for some $i \in [2, s]$ we have $\{i, 2s+1-i\} \notin \mathcal{F}$, then by Lemma~\ref{l.EG} we get $y_{\mathcal{F}}(2) \geq \frac{(4s-3i+3)(i+2)}{2}$ and by convexity $y_{\mathcal{F}}(2) \geq \min\big\{2(4s-3), \frac{(s+3)(s+2)}{2}\big\}$. Combining it with the inequality $y_{\mathcal{F}}(3) \geq 3s+1$, we get $$y_{\mathcal{F}}(2) + y_{\mathcal{F}}(3) \geq \min\Big\{11s - 5, \frac{s^2 + 11s + 8}{2}\Big\}.$$
	
	We have $\frac{s^2 + 13s + 8}{2} > \frac{s^2 + 9s + 20}{2}$ for $s \geq 4$. Moreover, we have $11s - 5 > 10s + 5$ for $s \geq 11$,  $11s - 5 > \frac{s^2+9s+20}{2}$ for $4 \leq s \leq 9$, and $11s - 5 = 10s + 5 =  \frac{s^2+9s+20}{2}$ for $s=10$. Using it, we get that
	$$y_{\mathcal{F}}(2) + y_{\mathcal{F}}(3) \geq \min\big\{y_{\mathcal{P}(s, \ell)}(2), y_{\mathcal{Q}(s, \ell)}(2) + y_{\mathcal{Q}(s, \ell)}(3)\big\},$$
	and equality is attained only if $s = 10$ and $\mathcal{F}$ contains all $2$-element sets except those that may be shifted to $\{2, 2s-1\}$, as well as all $3$-element sets except those that may be shifted to $\{s, 2s, 2s+1\}$. However, in this case $\mathcal{F}$ contains an $s$-matching $\{\{i, 2s-1-i\} : i \in [2, s-2]\} \cup \{\{1, 2s+2\}, \{s-1, 2s-2, 2s+1\}, \{s, 2s-1, 2s\} \}$. It implies that the last displayed inequality is strict, and so $|\mathcal{F}| < \max(|\mathcal{P}(s, \ell)|, |\mathcal{Q}(s, \ell)|)$. Our only assumption was that $\{i,2s+1-i\}\notin \m F$, and so in what follows we assume that this is not the case, i.e., $\{i, 2s+1-i\} \in \mathcal{F}$ for all $i \in [2, s]$.
	
	Since the sets $\{i, 2s+1-i\}, i \in [2, s]$ form an $(s-1)$-matching in $\mathcal{F}$ and $\nu(\mathcal{F}) < s$, then no subset of $[n] \setminus \bigcup_{i=2}^{s}\{i, 2s+1-i\} = \{1, 2s, 2s+1, 2s+2\}$ can lie in $\mathcal{F}$. By shiftedness, it implies $\mathcal{F} \subset \mathcal{W}(s, \ell)$ and therefore $y_{\mathcal{F}}(2) + y_{\mathcal{F}}(3) \geq y_{\mathcal{W}(s,\ell)}(2) + y_{\mathcal{W}(s,\ell)}(3)$. Since $y_{\mathcal{W}(s,\ell)}(2) + y_{\mathcal{W}(s,\ell)}(3) = 12s-2 > 10s+5$ for $s \geq 4$, we get $|\mathcal{F}| < |\mathcal{Q}(s, \ell)|$.
\end{proof}

Proofs for $c=1,3,4$ are very similar to the proof for $c=2$. We only have two notable differences. First, for $c=1$ we have $\mathcal{W}(s,\ell)=\mathcal{Q}(s, \ell)$ and therefore in the last step of the proof we get another extremal family. Second, for $c=3$ and $c=4$ we may only consider  cases $d(\mathcal{F}) \in \{3, 4\}$, since other cases were treated in Section~\ref{s.cases}.

Since the rest of the proof is very similar to the proof of Lemma~\ref{l.c_eq_2} and we have already discussed all its ideas, we provide formal proofs almost without discussion.

\subsection{$c=1$}

\begin{lemma} \label{l.c_eq_1}
	Let $n = 2s+1$ and $\mathcal{F} \subset 2^{[n]}$ is a shifted up-set with $\nu(\mathcal{F}) < s$ and $\mathcal{F} \cap {[n] \choose 1} = \emptyset$. Then $|\mathcal{F}| \leq \max\{|\mathcal{P}(s, \ell)|, |\mathcal{P}'(s, \ell)|, |\mathcal{Q}(s, \ell)|, |\mathcal{W}(s, \ell)|\}$. Moreover, equality is achieved only if $\mathcal{F}$ is one of the families $\mathcal{P}(s, \ell), \mathcal{P}'(s, \ell),$ $\mathcal{Q}(s, \ell),$ $\mathcal{W}(s, \ell)$.
\end{lemma}

Again, for $c=1$ the maximum cannot be attained in the second argument, but it is still more convenient to prove $|\mathcal{F}| \leq \max\{|\mathcal{P}(s, \ell)|, |\mathcal{P}'(s, \ell)|,$ $|\mathcal{Q}(s, \ell)|,$ $|\mathcal{W}(s, \ell)|\}$ than $|\mathcal{F}| \leq \max\{|\mathcal{P}(s, \ell)|,$ $|\mathcal{Q}(s, \ell)|,$ $|\mathcal{W}(s, \ell)|\}$.

\begin{proof}
	The case $\ell \leq 1$, that is, $s \leq 2$ is treated in \cite{Kl} and the case $\ell = 2$, that is, $s = 3$ is treated in \cite{FK9}, so we will assume $s \geq 4$.
	
	For $c = 1$ we have
	$$y_{\mathcal{P}(s, \ell)}(2) = \frac{s^2 + 5s + 6}{2}$$
	and
	$$y_{\mathcal{Q}(s, \ell)}(2) + y_{\mathcal{Q}(s, \ell)}(3) = y_{\mathcal{W}(s, \ell)}(2) + y_{\mathcal{W}(s, \ell)}(3) = 6s-2.$$
	
	Since $\mathcal{F}$ does not contains any singletons, the inequality $$|\mathcal{F}| \leq \max\big\{|\mathcal{P}(s, \ell)|, |\mathcal{P}'(s, \ell)|, |\mathcal{Q}(s, \ell)|, |\mathcal{W}(s, \ell)|\big\}$$ is a consequence of 
    {\small $$y_{\mathcal{F}}(2) + y_{\mathcal{F}}(3) \geq \min\big\{y_{\mathcal{P}(s, \ell)}(2), y_{\mathcal{P}'(s, \ell)}(2), y_{\mathcal{Q}(s, \ell)}(2) + y_{\mathcal{Q}(s, \ell)}(3), y_{\mathcal{W}(s, \ell)}(2) + y_{\mathcal{W}(s, \ell)}(3)\big\}.$$}
	If $d(\mathcal{F}) = 0$, then by Lemma \ref{l.even_d_to_y2} we get that $y_{\mathcal{F}}(2) \geq \max\{y_{\mathcal{P}(s, \ell)}(2), y_{\mathcal{P}'(s, \ell)}(2)\}$ and equality is achieved only if $\mathcal{F}$ is one of the families $\mathcal{P}(s, \ell)$ and $\mathcal{P}'(s, \ell)$. Therefore, we assume $d(\mathcal{F}) > 0$. By the definition of $d(\mathcal{F})$ it implies that $\{i, 2s-i-1\} \in \mathcal{F}$ for all $i \in [s-1]$. Since these sets together with a set $\{2s-1, 2s, 2s+1\}$ form an $s$-matching and $\nu(\mathcal{F}) < s$, we have $\{2s-1, 2s, 2s+1\} \notin \mathcal{F}$.
	
	If $\{1, 2s - 1\} \notin \mathcal{F}$, then by shiftedness of $\mathcal{F}$ we have $\mathcal{F} \subset \mathcal{Q}(s, \ell)$. Therefore, $|\mathcal{F}| \leq |\mathcal{Q}(s, \ell)|$ and equality is attained only if $\mathcal{F} = \mathcal{Q}(s, \ell)$. If the rest of the proof we assume $\{1, 2s-1\} \in \mathcal{F}$.
	
	Next, we deal with the case $\{i, 2s-i\} \notin \mathcal{F}$ for some $i \in [2, s-2]$. By Lemma~\ref{l.EG} it implies $y_{\mathcal{F}}(2) \geq \frac{(4s-3i+1)(i+2)}{2}$  for some $i \in [2, s - 2]$ and by convexity we get $y_{\mathcal{F}}(2) \geq \min(8s-10, \frac{s^2 + 7s}{2})$. Since $s \geq 4$, we get
	\begin{align*}y_{\mathcal{F}}(2) &\geq \min\Big\{8s-10, \frac{s^2+7s}{2}\Big\} > \min\Big\{6s-2, \frac{s^2 + 5s + 6}{2}\Big\} \\
    &= \min\big\{y_{\mathcal{Q}(s, \ell)}(2) + y_{\mathcal{Q}(s, \ell)}(3), y_{\mathcal{P}(s, \ell)}(2)\big\}.\end{align*}
    The second inequality is correct since for $s \geq 5$ we have $8s-10 > 6s-2$ and $\frac{s^2+7s}{2} > \frac{s^2 + 5s + 6}{2}$ and for $s=4$ we get $$\min\Big\{8s-10, \frac{s^2+7s}{2}\Big\} = 32 > 31 = \min\Big\{6s-2, \frac{s^2 + 5s + 6}{2}\Big\}.$$
	Therefore, $|\mathcal{F}| < \max\big\{|\mathcal{Q}(s, \ell)|, |\mathcal{P}(s, \ell)|\big\}$. In the rest of the proof we assume $\{i, 2s-i\} \in \mathcal{F}$ for all $i \in [2, s-2]$.
	
	Consider an $s$-matching $\{\{1, 2s - 1\}, \{2, 2s-2\}, \ldots, \{s - 2, s + 2\}, \{s - 1, s\}, \{s + 1, 2s, 2s+1\}\}$. Since all $2$-element sets in this matching are in $\m F$ and $\nu(\mathcal{F}) < s$,  we get that $\{s + 1, 2s, 2s+1\} \notin \mathcal{F}$. By shiftedness, we get $\{x, 2s, 2s+1\} \notin \mathcal{F}$ for all $x \in [s + 1, 2s - 1]$ and, therefore, $y_{\mathcal{F}}(3) \geq s - 1$.
	
	If $\{s-1, s+1\} \notin \mathcal{F}$, by Lemma \ref{l.EG} we get $y_{\mathcal{F}}(2) \geq \frac{(s+4)(s+1)}{2} = \frac{s^2 + 5s + 4}{2}$. Combining this bound with $y_{\mathcal{F}}(3) \geq s - 1$, we get
	$$y_{\mathcal{F}}(2) + y_{\mathcal{F}}(3) \geq \frac{s^2 + 7s + 2}{2} > \frac{s^2 + 5s + 6}{2} = y_{\mathcal{P}(s, \ell)}(2)$$
	and, therefore, $|\mathcal{F}| < |\mathcal{P}(s, \ell)|$. In what follows, we assume $\{s-1, s+1\} \in \mathcal{F}$.

	Since $\{i, 2s-i\} \in \mathcal{F}$ for $i \in [s - 1]$ and these sets together with $\{s, 2s, 2s+1\}$ form an $s$-matching, we get $\{s, 2s, 2s+1\} \notin \mathcal{F}$. By shiftedness it implies 
    \begin{align} \label{eq:y3_geq_s}
        y_{\mathcal{F}}(3) \geq s.
    \end{align}
	
	Next, consider sets $\{i, 2s+1-i\}$ for $i \in [2, s-1]$. If one of these sets is not in $\mathcal{F}$, by Lemma \ref{l.EG} and convexity we get $y_{\mathcal{F}}(2) \geq \min\big\{6s-6, \frac{(s+5)s}{2}\big\}$. Combining this bound with $y_{\mathcal{F}}(3) \geq s$ and using $s \geq 4$, we get
	\[
	\begin{split}
		y_{\mathcal{F}}(2) + y_{\mathcal{F}}(3) &\geq \min\Big\{7s-6, \frac{s^2+7s}{2}\Big\} \\& > \min\Big\{6s-2, \frac{s^2 + 5s + 6}{2}\Big\} \\
        &= \min\big\{y_{\mathcal{Q}(s, \ell)}(2) + y_{\mathcal{Q}(s, \ell)}(3), y_{\mathcal{P}(s, \ell)}(2)\big\}.
	\end{split}
	\]
    Again, the second inequality is obvious for $s \geq 5$ and the case $s=4$ may be treated separately.
	Therefore, $|\mathcal{F}| < \max\big\{|\mathcal{P}(s, \ell)|, |\mathcal{Q}(s, \ell)|\big\}$. Thus, we may assume $\{i, 2s+1-i\} \in \mathcal{F}$ for all $i \in [2, s-1]$.

    If the set $\{s, s+1\}$ also lies if $\mathcal{F}$, then the sets $\{i, 2s+1-i\}, i \in [2, s]$, form an $(s-1)$-matching in $\mathcal{F}$ that leaves the elements $1, 2s, 2s+1$ uncovered. Therefore, we have $\{1, 2s, 2s+1\} \notin \mathcal{F}$ and $\{1, 2s\} \notin \mathcal{F}$. Shiftedness of $\mathcal{F}$ implies $\mathcal{F} \subset \mathcal{W}(s, \ell)$. Thus, we have $|\mathcal{F}| \leq |\mathcal{W}(s, \ell)|$ and equality is attained only if $\mathcal{F} = \mathcal{W}(s, \ell)$.

    The only remaining case is as follows:  $\{i, 2s+1-i\} \in \mathcal{F}$ for $i \in [2, s-1]$, but $\{s, s+1\}\notin\mathcal{F}$. We will show that in this case $|\mathcal{F}| < |\mathcal{P}(s, \ell)|$. We consider two more subcases: whether $\{1, 2s\} \in \mathcal{F}$ or $\{1, 2s\} \notin \mathcal{F}$.

    If $\{1, 2s\} \in \mathcal{F}$ then $\{s, s+1, s+2\} \notin \mathcal{F}$, since otherwise $\{\{1, 2s\}, \ldots, \{s - 1, s + 2\}, \{s, s + 1, 2s+1\}\}$ is an $s$-matching in $\mathcal{F}$. By shiftedness we have $\{x, y, 2s+1\} \notin \mathcal{F}$ for all $x,  y \in [s, 2s]$, consequently $y_{\mathcal{F}}(3) \geq {s + 1 \choose 2}$. Since $\{s, s+1\} \notin \mathcal{F}$, Lemma \ref{l.EG} implies that $y_{\mathcal{F}}(2) \geq \frac{(s+2)(s+1)}{2}$. Combining these bounds, we get
    $$y_{\mathcal{F}}(2) + y_{\mathcal{F}}(3) \geq \frac{2s^2 + 4s + 2}{2} > \frac{s^2 + 5s + 6}{2} = y_{\mathcal{P}(s, \ell)}(2).$$

    Finally, consider the case $\{1, 2s\}, \{s, s + 1\} \notin \mathcal{F}$. In this case, $\mathcal{F}$ does not contain any sets that can be shifted to $\{1, 2s\}$ or to $\{s, s+1\}$. There are ${s + 2 \choose 2}$ sets that can be shifted to $\{s, s+1\}$ and $2(s - 1)$ sets that can be shifted to $\{1, 2s\}$, but cannot be shifted to $\{s, s+1\}$. Thus, $y_{\mathcal{F}}(2) \geq {s + 2 \choose 2} + 2(s - 1) = \frac{s^2 + 7s - 2}{2}$. Combining this bound with \eqref{eq:y3_geq_s}, we get
    $$y_{\mathcal{F}}(2) + y_{\mathcal{F}}(3) \geq \frac{s^2 + 9s - 2}{2} > \frac{s^2 + 5s + 6}{2} = y_{\mathcal{P}(s, \ell)}(2).$$
	
\end{proof}

\subsection{$c \in \{3, 4\}$}
\begin{lemma} \label{l.c_eq_3}
	Let $n = 2s+3$ and $\mathcal{F} \subset 2^{[n]}$ is a shifted up-set with $\nu(\mathcal{F}) < s$ and $\mathcal{F} \cap {[n] \choose 1} = \emptyset$. Then $|\mathcal{F}| \leq \max(|\mathcal{P}(s, \ell)|, |\mathcal{P}'(s, \ell)|, |\mathcal{Q}(s, \ell)|)$. Moreover, equality is achieved only if $\mathcal{F}$ is one of the families $\mathcal{P}(s, \ell), \mathcal{P}'(s, \ell), \mathcal{Q}(s, \ell)$.
\end{lemma}

\begin{proof}
	The case $\ell \leq 1$, that is, $s \leq 4$, is treated in \cite{Kl}, so we will assume $s \geq 5$.
	
	For $c = 3$ we have
	$$y_{\mathcal{P}(s, \ell)}(2) = \frac{s^2 + 13s + 42}{2}$$
	and
	$$y_{\mathcal{Q}(s, \ell)}(2) + y_{\mathcal{Q}(s, \ell)}(3) = 14s+28.$$
	
	Since $\mathcal{F}$ does contains no singletons, the inequality $$|\mathcal{F}| \leq \max\big\{|\mathcal{P}(s, \ell)|, |\mathcal{P}'(s, \ell)|, |\mathcal{Q}(s, \ell)|\big\}$$ is a consequence of $$y_{\mathcal{F}}(2) + y_{\mathcal{F}}(3) \geq \min\left\{y_{\mathcal{P}(s, \ell)}(2), y_{\mathcal{P}'(s, \ell)}(2), y_{\mathcal{Q}(s, \ell)}(2) + y_{\mathcal{Q}(s, \ell)}(3)\right\}.$$	
	If $d(\mathcal{F}) = 0$, then by Lemma \ref{l.even_d_to_y2} we get that $y_{\mathcal{F}}(2) \geq \max\left\{y_{\mathcal{P}(s, \ell)}(2), y_{\mathcal{P}'(s, \ell)}(2)\right\}$ and equality is achieved only if $\mathcal{F}$ is one of the families $\mathcal{P}(s, \ell)$ and $\mathcal{P}'(s, \ell)$.
	
	The case $d(\mathcal{F}) \in \{1, 2\}$ is treated in Lemma \ref{l.d_1_2} and the case $d(\mathcal{F}) \geq 5$ is treated in Lemma \ref{l.no_big_d}. Therefore, we may assume $d(\mathcal{F}) \in \{3, 4\}$. By the definition of $d(\mathcal{F})$, the condition $d > 2$ implies that $\mathcal{F}$ contains $\{i, 2s-3-i\}$ for all $i \in [s - 2]$ and by Lemma \ref{l.even_d_to_y2} the condition $d \leq 4$ implies
	{\small \begin{equation}\label{eqlem37} y_{\mathcal{F}}(2) \geq \min\left\{3(4s-1), \frac{(s+5)(s+6)}{2}\right\} = \min\left\{12s-3, \frac{s^2 + 11s + 30}{2}\right\}.\end{equation}}
	
	If $\{1, 2s-3\} \notin \mathcal{F}$ and $\mathcal{F} \cap {[2s-3, n] \choose 3} = \emptyset$, then $\mathcal{F} \subset \mathcal{Q}(s, \ell)$ and, therefore, $|\mathcal{F}| \leq |\mathcal{Q}(s, \ell)|$ and equality is attained only if $\mathcal{F} = \mathcal{Q}(s, \ell)$.
	
	If $\{1, 2s-3\} \notin \mathcal{F}$ and $\mathcal{F} \cap {[2s-3, n] \choose 3} \neq \emptyset$, we use Lemma \ref{l.old_d_to_y3} for $d = 2$ and get $y_{\mathcal{F}}(3) \geq 3s + 7$. Since $\{1, 2s - 3\} \notin \mathcal{F}$, by Lemma \ref{l.EG} we get $y_{\mathcal{F}}(2) \geq 14s-7$. Combining this bound with the bounds on $y(3)$, we get $$y_{\mathcal{F}}(2) + y_{\mathcal{F}}(3) \geq 17s.$$
	To conclude that $|\mathcal{F}| < \max\left\{|\mathcal{P}(s, \ell)|, |\mathcal{Q}(s, \ell)|\right\}$ we need to check that $17s > \min\left\{\frac{s^2 + 13s + 42}{2}, 14s+28\right\}$ for $s \geq 5$. Indeed, we have $17s > 14s + 28$ for $s \geq 10$ and $17s > \frac{s^2 + 13s + 42}{2}$ for $5 \leq s \leq 9$. We got this conclusion in the case $\{1, 2s-3\} \notin \mathcal{F}$, and so in what follows we assume $\{1, 2s-3\} \in \mathcal{F}$.
	
	Next, we prove, than for all $i \in [2, s - 3]$ we have $\{i, 2s-2-i\} \in \mathcal{F}$. Indeed, otherwise by Lemma \ref{l.EG} we have $y_{\mathcal{F}}(2) \geq \frac{(4s-3i+1)(i + 6)}{2}$ for some $i \in [2, s - 3]$ and by convexity $y_{\mathcal{F}}(2) \geq \min\left\{4(4s-5), \frac{(s + 10)(s + 3)}{2}\right\}$. By Lemma~\ref{l.d_geq_1_to_y3} we have $y_{\mathcal{F}}(3) \geq 28$. Combining these bounds, we get $y_{\mathcal{F}}(2) + y_{\mathcal{F}}(3) \geq \min\left\{16s + 8, \frac{s^2 + 13s + 86}{2}\right\}$. To conclude that $|\mathcal{F}| < \max\left\{|\mathcal{P}(s, \ell)|, |\mathcal{Q}(s, \ell)|\right\}$ we need to check that $\min\left\{16s + 8, \frac{s^2 + 13s + 86}{2}\right\} > \min\left\{\frac{s^2 + 13s + 42}{2}, 14s+28\right\}$ for $s \geq 5$, that is, for any $s$ each argument in the left-hand side is greater than at least one argument in the right-hand side. Indeed, we have $\frac{s^2 + 13s + 86}{2} > \frac{s^2 + 13s + 42}{2}$ for all $s$, $16s + 8 > 14s + 28$ for $s \geq 11$ and $16s + 8 > \frac{s^2 + 13s + 42}{2}$ for $5 \leq s \leq 10$.
	
	We have proved that if $\{i, 2s-2-i\} \notin \mathcal{F}$ for some $i \in [1, s - 3]$, then $|\mathcal{F}| < \max\left\{|\mathcal{P}(s, \ell)|, |\mathcal{Q}(s, \ell)|\right\}$. Therefore, we may assume $\{i, 2s-2-i\} \in \mathcal{F}$ for all $i \in [1, s - 3]$. Thus, we may apply Lemma \ref{l.old_d_to_y3} for $d=3$ and get $y_{\mathcal{F}}(3) \geq 5s + 1$. Finally, we combine this bound with the bound \eqref{eqlem37} on $y_{\mathcal{F}}(2)$ and get
	$$y_{\mathcal{F}}(2) + y_{\mathcal{F}}(3) \geq \min\left\{17s-2, \frac{s^2 + 21s + 32}{2}\right\}.$$
	We have $\frac{s^2 + 21s + 32}{2} > \frac{s^2 + 13s + 42}{2}$ for $s \geq 5$, $17s-2 > 14s+28$ for $s \geq 11$ and $17s-2 > \frac{s^2 + 13s + 42}{2}$ for $5 \leq s \leq 10$. Therefore,
	\[
	\begin{split}
		&y_{\mathcal{F}}(2) + y_{\mathcal{F}}(3) \geq \min\left\{17s-2, \frac{s^2 + 21s + 32}{2}\right\} > \\& > \min\left\{\frac{s^2 + 13s + 42}{2}, 14s+28\right\} = \min\left\{y_{\mathcal{P}(s, \ell)}(2), y_{\mathcal{Q}(s, \ell)}(2) + y_{\mathcal{Q}(s, \ell)}(3)\right\}.
	\end{split}
	\]
\end{proof}

\begin{lemma} \label{l.c_eq_4}
	Let $n = 2s+4$. Suppose that $\mathcal{F} \subset 2^{[n]}$ is a shifted up-set with $\nu(\mathcal{F}) < s$ and $\mathcal{F} \cap {[n] \choose 1} = \emptyset$. Then $|\mathcal{F}| \leq \max(|\mathcal{P}(s, \ell)|, |\mathcal{P}'(s, \ell)|, |\mathcal{Q}(s, \ell)|)$. Moreover, equality is achieved only if $\mathcal{F}$ is one of the families $\mathcal{P}(s, \ell), \mathcal{P}'(s, \ell), \mathcal{Q}(s, \ell)$.
\end{lemma}

\begin{proof}
	The case $\ell \leq 1$, that is, $s \leq 5$ is treated in \cite{Kl} and the case $\ell = 2$, that is $s = 6$ is treated in \cite{FK9}, so we will assume $s \geq 7$.
	
	For $c = 4$ we have
	$$y_{\mathcal{P}(s, \ell)}(2) = \frac{s^2 + 17s + 72}{2}$$
	and
	$$y_{\mathcal{Q}(s, \ell)}(2) + y_{\mathcal{Q}(s, \ell)}(3) = 18s+75.$$
	
	Since $\mathcal{F}$ contains no singletons, the inequality $$|\mathcal{F}| \leq \max(|\mathcal{P}(s, \ell)|, |\mathcal{P}'(s, \ell)|, |\mathcal{Q}(s, \ell)|)$$ is a consequence of $$y_{\mathcal{F}}(2) + y_{\mathcal{F}}(3) \geq \min(y_{\mathcal{P}(s, \ell)}(2), y_{\mathcal{P}'(s, \ell)}(2), y_{\mathcal{Q}(s, \ell)}(2) + y_{\mathcal{Q}(s, \ell)}(3)).$$	
	If $d(\mathcal{F}) = 0$, then by Lemma \ref{l.even_d_to_y2} we get that $y_{\mathcal{F}}(2) \geq \max(y_{\mathcal{P}(s, \ell)}(2), y_{\mathcal{P}'(s, \ell)}(2))$ and equality is achieved only if $\mathcal{F}$ is one of the families $\mathcal{P}(s, \ell)$ and $\mathcal{P}'(s, \ell)$.
	
	The case $d(\mathcal{F}) \in \{1, 2\}$ is treated in Lemma \ref{l.d_1_2} and the case $d(\mathcal{F}) \geq 5$ is treated in Lemma \ref{l.no_big_d}. Therefore, we may assume $d(\mathcal{F}) \in \{3, 4\}$. By definition of $d(\mathcal{F})$ the condition $d > 2$ implies that $\mathcal{F}$ contains $\{i, 2s-3-i\}$ for all $i \in [s - 2]$ and by Lemma \ref{l.even_d_to_y2} the condition $d \leq 4$ implies
    {\small \begin{align} \label{eq:y2_c4}
	y_{\mathcal{F}}(2) \geq \min\left\{9(2s-1), \frac{(s+7)(s+6)}{2}\right\} = \min\left\{18s-9, \frac{s^2 + 13s + 42}{2}\right\}.
    \end{align}}
	
	First, we deal with the case $\{i, 2s - 3 - i\} \notin \mathcal{F}$ for some $i \in [2, s-4]$. By Lemma \ref{l.EG} it implies $y_{\mathcal{F}} \geq \frac{(4s - 3i + 1)(i + 8)}{2}$. By convexity we get
	$$y_{\mathcal{F}}(2) \geq \min\left\{5(4s-5), \frac{(s+13)(s+4)}{2}\right\} = \min\left\{20s-25, \frac{s^2+17s+52}{2}\right\}.$$
	Combining it with a bound $y_{\mathcal{F}}(3) \geq 55$ from Lemma \ref{l.d_geq_1_to_y3}, we get $y_{\mathcal{F}}(2) + y_{\mathcal{F}}(3) \geq \min\left\{20s+30, \frac{s^2+17s+162}{2}\right\}$. Since $\frac{s^2+17s+162}{2} > \frac{s^2 + 17s + 72}{2}$ for all $s$, $20s+30 > 18s+75$ for $s \geq 23$ and $20s+30 > \frac{s^2 + 17s + 72}{2}$ for $6 \leq s \leq 22$, we get
	\[
	\begin{split}
		y_{\mathcal{F}}(2) + y_{\mathcal{F}}(3) &\geq \min\left\{20s+30, \frac{s^2+17s+162}{2}\right\} \\& > \min\left\{\frac{s^2 + 17s + 72}{2}, 18s+75\right\} \\
        &= \min\left\{y_{\mathcal{P}(s, \ell)}(2), y_{\mathcal{Q}(s, \ell)}(2) + y_{\mathcal{Q}(s, \ell)}(3)\right\}.
	\end{split}
	\]
	Thus, we have proved that if $\{i, 2s - 3 - i\} \notin \mathcal{F}$ for some $i \in [2, s-4]$, then $|\mathcal{F}| < \max\left\{|\mathcal{P}(s, \ell)|, |\mathcal{Q}(s, \ell)|\right\}$. In what follows, we assume $\{i, 2s - 3 - i\} \in \mathcal{F}$ for all $i \in [2, s-4]$.
	
	If $\{1, 2s-4\} \notin \mathcal{F}$ and $\mathcal{F} \cap {[2s-4, n] \choose 3} = \emptyset$, then $\mathcal{F} \subset \mathcal{Q}(s, \ell)$ and, therefore, $|\mathcal{F}| \leq |\mathcal{Q}(s, \ell)|$ and equality is attained only if $\mathcal{F} = \mathcal{Q}(s, \ell)$.
	
	If $\{1, 2s-4\} \notin \mathcal{F}$ and $\mathcal{F} \cap {[2s-4, n] \choose 3} \neq \emptyset$, we use Lemma \ref{l.old_d_to_y3} for $d = 3$ and get $y_{\mathcal{F}}(3) \geq 5s + 11$. Since $\{1, 2s - 4\} \notin \mathcal{F}$, by Lemma \ref{l.EG} we get $y_{\mathcal{F}}(2) \geq 18s-9$. Combining these bounds, we get
    $$y_{\mathcal{F}}(2) + y_{\mathcal{F}}(3) \geq 23s + 2.$$
	To conclude that $|\mathcal{F}| < \max\left\{|\mathcal{P}(s, \ell)|, |\mathcal{Q}(s, \ell)|\right\}$ we need to check that $23s + 2 > \min\left\{\frac{s^2 + 17s + 72}{2}, 18s+75\right\}$ for $s \geq 6$. Indeed, we have $23s + 2 > 18s+75$ for $s \geq 15$ and $23s + 2 > \frac{s^2 + 17s + 72}{2}$ for $6 \leq s \leq 14$. Since we considered the case $\{1, 2s-4\} \notin \mathcal{F}$, in futher proof we assume $\{1, 2s-4\} \in \mathcal{F}$.
	
	Finally, we use Lemma \ref{l.old_d_to_y3} for $d=4$ and get $y_{\mathcal{F}}(3) \geq 7s+1$. Combining it with the bound \ref{eq:y2_c4}, we get $y_{\mathcal{F}}(2) + y_{\mathcal{F}}(3) \geq \min\left\{25s-8, \frac{s^2+27s+44}{2}\right\}$. We have $\frac{s^2+27s+44}{2} > \frac{s^2 + 17s + 72}{2}$ for $s \geq 6$, $25s-8 > 18s+75$ for $s \geq 12$ and $25s-8 > \frac{s^2 + 17s + 72}{2}$ for $6 \leq s \leq 11$. Therefore, we get
	\[
	\begin{split}
		&y_{\mathcal{F}}(2) + y_{\mathcal{F}}(3) \geq \min\left\{25s-8, \frac{s^2+27s+44}{2}\right\} > \\& > \min\left\{\frac{s^2 + 17s + 72}{2}, 18s+75\right\} = \min\left\{y_{\mathcal{P}(s, \ell)}(2), y_{\mathcal{Q}(s, \ell)}(2) + y_{\mathcal{Q}(s, \ell)}(3)\right\}
	\end{split}
	\]
	and thus $|\mathcal{F}| < \max\left\{|\mathcal{P}(s, \ell)|, |\mathcal{Q}(s, \ell)|\right\}$.
\end{proof}

\end{document}